\documentclass[11pt,preprint]{imsart}
\RequirePackage[numbers]{natbib}
\RequirePackage[colorlinks,citecolor=blue,urlcolor=blue]{hyperref}
\usepackage{amsmath,amsmath,amssymb,amsfonts}
\usepackage{amsthm}
\usepackage[american]{babel}
\usepackage{graphicx}
\usepackage{flafter}
\usepackage[section]{placeins}
\usepackage{float}
\usepackage{makecell}
\usepackage{comment}
\usepackage[mathscr]{eucal}
\usepackage{bbm}
\usepackage{todonotes}
\presetkeys{todonotes}{color=red!30}{linecolor=black!50}
\startlocaldefs

\def\E{{\mathbb E }}

\def\eps{{\varepsilon}}

\def\Bbb E{\mathbb{E}}
\def\Bbb R{\mathbb{R}}
\newtheorem{example}{Example}
\newtheorem{corollary}{Corollary}[section]
\usepackage{mathtools}
\makeatletter \@addtoreset{equation}{section}

\makeatother

\newtheorem{lemma}{Lemma}[section]
\newtheorem{theorem}{Theorem}[section]
\newtheorem{proposition}{Proposition}[section]

\newtheorem{remark}{Remark}[section]

\newcommand\inner[2]{\langle #1, #2 \rangle}

\font\tencmmib=cmmib10 \skewchar\tencmmib '60
\newfam\cmmibfam
\textfont\cmmibfam=\tencmmib

\font\tenmsb=msbm10 
\def\Bbb#1{\hbox{\tenmsb#1}}

\def\lessim{\ \lower4pt\hbox{$
\buildrel{\displaystyle <}\over\sim$}\ }
\def\gessim{\ \lower4pt\hbox{$\buildrel{\displaystyle >}
\over\sim$}\ }

\def\eps{\varepsilon}

\def\go0{\to 0}

\def\leftitem#1{\item{\hbox to\parindent{\enspace#1\hfill}}}

\def\QED{{$\hfill \bbox$}}

\def\sg{\sigma}

\def\sg2{\sigma^2}

\def\__{_{\infty}}
\def \EB{E}

\def \T {\mathcal{T}}
\def \B {\mathcal{B}}

\def \I {\mathcal{I}}
\def \Ind{I}
\def \QED{\hfill$\square$}

\numberwithin{equation}{section} 

\newcommand{\1}{{\rm 1}\kern-0.24em{\rm I}}
\def\E{\mathbb E}
\def\R{\mathbb R}
\DeclareMathOperator{\spann}{span}
\newtheorem{assumption}{Assumption}

\endlocaldefs

\begin{document}

\begin{frontmatter}
\title{Efficient Estimation of Smooth Functionals in Gaussian Shift Models}
\runtitle{Estimation of smooth functionals}

\begin{aug}
\author{\fnms{Vladimir} \snm{Koltchinskii}\thanksref{t1}\ead[label=e1]{vlad@math.gatech.edu}} and 
\author{\fnms{Mayya} \snm{Zhilova}\thanksref{m1}\ead[label=e2]{mzhilova@math.gatech.edu}}
\thankstext{t1}{Supported in part by NSF Grants DMS-1810958, DMS-1509739 and CCF-1523768}
\thankstext{m1}{Supported by NSF Grant DMS-1712990}
\runauthor{V. Koltchinskii and M. Zhilova}

\affiliation{Georgia Institute of Technology\thanksmark{m1}}

\address{School of Mathematics\\
Georgia Institute of Technology\\
Atlanta, GA 30332-0160\\
\printead{e1}\\
\printead*{e2}
}
\end{aug}
\vspace{0.2cm}
{\small \today}
\vspace{0.2cm}

\begin{abstract}
We study a problem of estimation of smooth functionals of parameter $\theta $ of Gaussian shift model 
$$
X=\theta +\xi,\ \theta \in E,
$$
where $E$ is a separable Banach space and $X$ is an observation of unknown vector $\theta$ 
in Gaussian noise $\xi$ with zero mean and known covariance operator $\Sigma.$ 
In particular, we develop estimators $T(X)$ of $f(\theta)$ for functionals $f:E\mapsto {\mathbb R}$ 
of H\"older smoothness $s>0$ such that 
$$
\sup_{\|\theta\|\leq 1} {\mathbb E}_{\theta}(T(X)-f(\theta))^2 \lesssim \Bigl(\|\Sigma\| \vee ({\mathbb E}\|\xi\|^2)^s\Bigr)\wedge 1,
$$
where $\|\Sigma\|$ is the operator norm of $\Sigma,$
and show that this mean squared error rate is minimax optimal at least in the case 
of standard Gaussian shift model ($E={\mathbb R}^d$ equipped with the canonical Euclidean norm, 
$\xi =\sigma Z,$ $Z\sim {\mathcal N}(0;I_d)$). Moreover, we determine a sharp threshold on the smoothness 
$s$ of functional $f$ such that, for all $s$ above the threshold, $f(\theta)$ can be estimated efficiently
with a mean squared error rate of the order $\|\Sigma\|$  
in a ``small noise" setting (that is, when ${\mathbb E}\|\xi\|^2$ is small). 
The construction of efficient estimators is crucially based on a ``bootstrap chain" method
of bias reduction. The results could be applied to a variety of special high-dimensional and 
infinite-dimensional Gaussian models (for vector, matrix and functional data).   
\end{abstract}



\end{frontmatter}


\newpage


\section{Introduction}
\label{Intro}

The problem of estimation of functionals of ``high complexity" parameters of statistical models often occurs 
both in high-dimensional and in nonparametric statistics, where it is of importance to identify some features 
of a complex parameter that could be estimated efficiently with a fast (sometimes, parametric) convergence 
rates. Such problems are very important in the case of vector, matrix or functional parameters in a variety of 
applications including functional data analysis and kernel machine learning 
(\cite{RamsaySilverman}, \cite{BlanchardBousquetZwald}). In this paper, we study a very basic version 
of this problem in the case of rather general Gaussian models with unknown mean.  
Consider the following {\it Gaussian shift model}
\begin{align}
\label{Gauss_shift}
X=\theta + \xi,\  \theta \in E,
\end{align}
where $E$ is a separable Banach space, $\theta$ is an unknown parameter and $\xi$ is a mean zero Gaussian random variable in $E$
(the noise) with {\it known} covariance operator $\Sigma.$ 
In other words, an observation $X\sim {\mathcal N}(\theta;\Sigma)$ in Gaussian shift model \eqref{Gauss_shift} 
is a Gaussian vector in $E$ with unknown mean $\theta$ and known covariance 
$\Sigma.$ 
Recall that $\Sigma$ is an operator from 
the dual space $E^{\ast}$ into $E$ such that $\Sigma u := {\mathbb E}\langle \xi,u\rangle \xi, u\in E^{\ast}.$
Here and in what follows, $\langle x, u\rangle$ denotes the value of a linear functional $u\in E^{\ast}$
on a vector $x\in E$ (although, in some parts of the paper, with a little abuse of notation, $\langle \cdot, \cdot \rangle$ will also 
denote the inner product of Euclidean spaces). It is well known that the covariance operator $\Sigma$ of a Gaussian 
vector in $E$ is bounded and, moreover, it is nuclear. 

Our goal is to study the problem of estimation of $f(\theta)$ for smooth functionals $f:E \mapsto {\mathbb R}.$
The problem of estimation of smooth functionals of parameters of infinite-dimensional (nonparametric) models 
has been studied for several decades. It is considerably harder than in the classical finite-dimensional 
parametric i.i.d. models, where under standard regularity assumptions, $f(\hat \theta)$ ($\hat \theta$ being the maximum 
likelihood estimator) is an asymptotically efficient (in the sense of H\`ajek-LeCam) estimator of $f(\theta)$ with 
$\sqrt{n}$-rate for continuously differentiable functions $f.$ In the nonparametric case, classical convergence 
rates do not necessarily hold in functional estimation problems and minimax optimal convergence rates
have to be determined. Moreover, even when the classical convergence rates do hold, the construction of efficient estimator is often a challenging problem. Such problems have been often studied for special models 
(Gaussian white noise model, nonparametric density estimation model, etc) and for special functionals
(with a number of nontrivial results even in the case of linear and quadratic functionals).
Early results in this direction are due to Levit \cite{Levit_1, Levit_2}
and Ibragimov and Khasminskii \cite{Ibragimov}. Further important references include Ibragimov, Nemirovski and 
Khasminskii \cite{Ibragimov_Khasm_Nemirov}, Donoho and Liu \cite{Donoho_1, Donoho_2}, Bickel and Ritov \cite{Bickel_Ritov}, Donoho and Nussbaum \cite{Donoho_Nussbaum}, Nemirovski \cite{Nemirovski_1990, Nemirovski}, Birg\'e and Massart \cite{Birge}, Laurent \cite{Laurent}, 
Lepski, Nemirovski and Spokoiny \cite{Lepski}, Cai and Low \cite{Cai_Low_2005a, Cai_Low_2005b}, Klemel\"a \cite{Klemela} as well as a vast literature on semiparametric efficiency (see, e.g., \cite{BKRW} and references therein).
Early results on consistent and asymptotically 
normal estimation of smooth functionals of high-dimensional parameters are due to Girko \cite{Girko, Girko-1}.  More recently, there has been a lot of interest in efficient and minimax optimal estimation of functionals of parameters of high-dimensional 
models including a variety of problems related to semiparametric efficiency of regularized estimators  
(see \cite{vdgBuhlmannRitovDezeureAOS}, \cite{Montanari}, \cite{Zhang_Zhang}), on minimax optimal 
rates of estimation of special functionals (see \cite{C_C_Tsybakov}), on efficient estimation of smooth 
functionals of covariance in Gaussian models \cite{Koltchinskii_Nickl, Koltchinskii_2017}.

Throughout the paper, given nonnegative $A,B,$ $A\lesssim B$ means that $A\leq CB$ for a numerical 
constant $C,$ $A\gtrsim B$ is equivalent to $B\lesssim A$ and $A\asymp B$ is equivalent to $A\lesssim B\lessim A.$
Sometimes signs of relationships $\lessim, \gtrsim$ and $\asymp$ will be provided with subscripts (say, $A\lesssim_{\gamma} B$ or $A\asymp_{\gamma} B$), indicating possible dependence of the constants on the corresponding parameters. 

In what follows, exponential bounds on random variables (say, on $\zeta$) are often stated in the following form: there exists a constant $C>0$ such that, for all $t\geq 1$ with probability at least $1-e^{-t},$ 
$\zeta\leq Ct.$ The proof could often result in a slightly different bound, for instance, $\zeta \leq Ct$ with probability 
$1-5e^{-t}.$ However, replacing constant $C$ with $C'= 2\log (5)C,$ it is easy to obtain the probability 
bound in the initial form $1-e^{-t}.$ In such cases, we say that ,``adjusting the constants" allows us to write the probability as $1-e^{-t}$ (without providing further details).

We will now briefly discuss the results of Ibragimov, Nemirovski and 
Khasminskii \cite{Ibragimov_Khasm_Nemirov} and follow up results 
of  Nemirovski \cite{Nemirovski_1990, Nemirovski} that are especially 
close to our approach to the problem. In \cite{Ibragimov_Khasm_Nemirov},
the following model was studied 
$$
dX^{(n)}(t)= \theta(t)dt +\frac{1}{\sqrt{n}}dw(t), t\in [0,1],
$$
in which a ``signal" $\theta \in \Theta\subset L_2([0,1])$ 
is observed in a Gaussian white noise ($w$ being a standard Brownian motion on $[0,1]$). 
The complexity of the parameter space $\Theta$ was characterized by Kolmogorov widths:
$$
d_m(\Theta):= \inf_{L\subset L_2([0,1]), {\rm dim}(L)\leq m} \sup_{\theta\in \Theta}\|\theta-P_L\theta\|^2,
$$
where $P_L$ denotes the orthogonal projection onto subspace $L.$ Assuming that $\Theta\subset U:=\{\theta\in L_2([0,1]): \|\theta\|\leq 1\}$ and, for some $\beta>0,$
$$
d_m(\Theta)\lesssim m^{-\beta}, m\geq 1,
$$
the goal of the authors was to determine a ``smoothness threshold" $s(\beta)>0$ such that, for all 
$s>s(\beta)$ and for all functionals $f$ on $L_2([0,1])$ of smoothness $s,$ $f(\theta)$ could be estimated 
efficiently with rate $n^{-1/2}$ based on observation $X^{(n)}$ (whereas for $s<s(\beta)$ there exist functionals 
$f$ of smoothness $s$ such that $f(\theta)$ could not be estimated with parametric rate $n^{-1/2}$).  It turned 
out that the main difficulties in this problem are related to a proper definition of the smoothness 
of the functional $f.$ In particular, even such simple functional as $f(\theta)=\|\theta\|^2$ could not be 
estimated efficiently on some sets $\Theta$ with $\beta\leq 1/4.$ The smoothness of functionals on Hilbert 
space $L_2([0,1])$ is usually defined in terms of their H\"older type norms that, in turn, depend on a way 
in which the norm of Fr\'echet derivatives $f^{(k)}(\theta)$ is defined. The $k$-th order Fr\'echet derivative 
is a symmetric $k$-linear form on $L_2([0,1]).$ The most common definition of the norm of such a form 
$M(h_1,\dots, h_k), h_1,\dots, h_k\in L_2([0,1])$ is the operator norm:
$
\|M\|:= \sup_{h_1,\dots, h_k\in U}|M(h_1,\dots, h_k)|.
$
Other possibilities include Hilbert--Schmidt norm $\|M\|_{HS}$ and 
``hybrid" norms $\|M\|_{(j)}:=\sup_{h_1,\dots, h_j\in U}\|M(h_1,\dots, h_j, \cdot, \dots, \cdot)\|_{HS}, 0\leq j\leq k.$
The H\"older classes in \cite{Ibragimov_Khasm_Nemirov} were defined in terms of the following 
norms: for $s=k+\gamma,$ $k\geq 0,$ $\gamma \in (0,1],$
$$
\|f\|_{\tilde C^s} := \max_{0\leq j\leq k-1}\sup_{\theta\in 2U} \|f^{(j)}(\theta)\|_{HS}\bigvee \sup_{\theta\in 2U}\|f^{(k)}(\theta)\|_{(1)}
\bigvee \sup_{\theta,\theta'\in 2U,\theta\neq \theta'}\frac{\|f^{(k)}(\theta)-f^{(k)}(\theta')\|}{\|\theta-\theta'\|}.
$$ 
With this somewhat complicated definition, it was proved that, if $\|f\|_{\tilde C^s}<\infty$ and,
either $k\leq 2$ and $s>\frac{1}{2\beta}+1,$ or $k\geq 3$ and $s>\frac{1}{2\beta},$ then there 
exists an asymptotically efficient estimator of $f(\theta)$ with convergence rate $n^{-1/2}.$ 
The construction of such estimators was based on the development of a method of unbiased 
estimation of Hilbert--Schmidt polynomials on $L_2([0,1])$ and on Taylor expansion of $f(\theta)$
in a neighborhood of an estimator $\hat \theta$ of $\theta$ with an optimal nonparametric error 
rate. It was later shown in \cite{Nemirovski_1990, Nemirovski} that the smoothness thresholds 
described above are optimal.

We will study similar problems for Gaussian shift model \eqref{Gauss_shift} trying to determine 
smoothness thresholds for efficient estimation in terms of proper complexity characteristics
for this model.  

Among the simplest smooth functionals on $E$ are bounded linear 
functionals $E\ni\theta\mapsto \langle \theta, u\rangle, u\in E^{\ast}.$ 
For a straightforward estimator $\langle X,u\rangle$ of such a functional,
$$
{\mathbb E}_{\theta}(\langle X,u\rangle- \langle\theta,u\rangle)^2= {\mathbb E}\langle \xi,u\rangle^2 
=\langle \Sigma u,u\rangle,
$$
and, for functionals $u$ from the unit ball of $E^{\ast}$ the largest possible mean squared 
error is equal to the operator norm of $\Sigma:$
$$
\|\Sigma\| = \sup_{\|u\|,\|v\|\leq 1} {\mathbb E}\langle \xi,u \rangle \langle \xi,v\rangle = 
\sup_{\|u\|\leq 1} {\mathbb E}\langle \xi,u\rangle^2.
$$
It is also not hard to prove the following proposition.

\begin{proposition}
Let 
$$
\hat T(X):=
\begin{cases}
\langle X,u\rangle\ {\rm for}\ \|\Sigma\|\leq 1\\
0\ \ \ \ \ \ \  {\rm for }\ \|\Sigma\|>1.
\end{cases}
$$
Then 
$$
\sup_{\|u\|\leq 1}\sup_{\|\theta\|\leq 1} {\mathbb E}_{\theta}(\hat T(X)-\langle \theta,u\rangle)^2 \leq 
\|\Sigma\| \wedge 1
$$
and 
\begin{equation}
\label{lower_MS}
\sup_{\|u\|\leq 1}\inf_{T}\sup_{\|\theta\|\leq 1} {\mathbb E}_{\theta}(T(X)-\langle \theta,u\rangle)^2 
\gtrsim 
\|\Sigma\|\wedge 1.
\end{equation}
\end{proposition}

In what follows, the complexity of estimation problem will 
be characterized by two parameters of the noise $\xi.$ One is the operator norm $\|\Sigma\|,$  
which is involved in the minimax mean squared error for estimation of linear functionals. 
It will be convenient to view $\|\Sigma\|$ as the {\it weak variance} of $\xi.$
Another complexity parameter is the {\it strong variance} of $\xi$ defined as 
$$
{\mathbb E}\|\xi\|^2 = {\mathbb E}\sup_{\|u\|,\|v\|\leq 1}\langle \xi,u \rangle \langle \xi,v\rangle
={\mathbb E}\sup_{\|u\|\leq 1} \langle \xi,u\rangle^2.
$$
Clearly, ${\mathbb E}\|\xi\|^2\geq \|\Sigma\|.$ The ratio of these two parameters, 
$${\bf r}(\Sigma):= \frac{{\mathbb E}\|\xi\|^2}{\|\Sigma\|},$$ 
is called the {\it effective rank} of $\Sigma$ and it was used earlier in concentration bounds 
for sample covariance operators and their spectral projections \cite{Koltchinskii_Lounici_arxiv, Koltchinskii_Lounici_bilinear}.    
The following properties 
of ${\bf r}(\Sigma)$ are obvious:
$$
{\bf r}(\Sigma) \geq 1\ {\rm and}\ {\bf r}(\lambda \Sigma)= {\bf r}(\Sigma), \lambda >0.
$$
Thus, the effective rank is invariant with respect to rescaling of $\Sigma$ (or rescaling of the noise).
In this sense, $\|\Sigma\|$ and ${\bf r}(\Sigma)$ can be viewed as complementary parameters 
of the noise. It is easy to check that, if $E$ is a Hilbert space, then ${\bf r}(\Sigma)=\frac{{\rm tr}(\Sigma)}{\|\Sigma\|},$
which implies that 
$
{\bf r}(\Sigma)\leq {\rm rank}(\Sigma)\leq {\rm dim}(E).
$ 
Clearly, ${\bf r}(\Sigma)$ could be viewed as a way to measure the dimensionality 
of the noise. In particular, for the maximum likelihood estimator $X$ of $\theta$ 
in the Gaussian shift model \eqref{Gauss_shift}, we have 
$
{\mathbb E}_{\theta}\|X-\theta\|^2 = {\mathbb E}\|\xi\|^2= \|\Sigma\| {\bf r}(\Sigma),
$
resembling a standard formula $\sigma^2 d$ for the risk of estimation of a vector 
in ${\mathbb R}^d$ observed in a ``white noise" with variance $\sigma^2.$ 

We discuss below several simple examples of the general Gaussian shift model \eqref{Gauss_shift}.

\begin{example}
\normalfont
\label{Example 1}
{\bf Standard Gaussian shift model.} Let $E={\mathbb R}^d$ be equipped with the canonical 
Euclidean inner product and the corresponding norm (the $\ell_2$-norm), and let $\xi=\sigma Z,$ where 
$\sigma>0$ is a known constant and $Z\sim {\mathcal N}(0;I_d).$ In this case, $\Sigma=\sigma^2 I_d,$ $\|\Sigma\|=\sigma^2,$
${\mathbb E}\|\xi\|^2 = \sigma^2 d$ and ${\bf r}(\Sigma)=d.$ Note that the size of effective rank ${\bf r}(\Sigma)$
crucially depends on the choice of underlying norm of the linear space. For instance, if 
$E={\mathbb R}^d={\ell}_{\infty}^d$ is equipped with the $\ell_{\infty}$-norm instead of $\ell_2$-norm,  
then we still have $\|\Sigma\|= \sigma^2,$ but 
$$
{\mathbb E}\|\xi\|_{\ell_{\infty}}^2 \asymp \sigma^2 \log d,
$$
implying that ${\bf r}(\Sigma)\asymp \log d.$
\end{example}

\begin{example}
\normalfont
\label{Example 2}
{\bf Matrix Gaussian shift models.} Let $E$ be the space of all symmetric $d\times d$ matrices 
equipped with the operator norm and let $\xi=\sigma Z$ with known parameter $\sigma>0$ and $Z$ sampled 
from the Gaussian orthogonal ensemble (that is, $Z=(Z_{ij})_{i,j=1}^d$ is a symmetric random matrix, 
$Z_{ij}, i\leq j$ are independent r.v., $Z_{ij}\sim {\mathcal N}(0,1), i<j,$ $Z_{ii}\sim {\mathcal N}(0;2)$).  
In this case, $\|\Sigma\|\asymp \sigma^2$ and 
$$
{\mathbb E}\|\xi\|^2 = \sigma^2 {\mathbb E}\|Z\|^2 \asymp \sigma^2 d,
$$
implying that ${\bf r}(\Sigma)\asymp d.$ As before, the effective rank would be different 
for a different choice of norm on $E.$ For instance, if $E$ is equipped with the Hilbert--Schmidt 
norm, then ${\bf r}(\Sigma)\asymp d^2$ (compare this with Example 1). 
\end{example}

\begin{example}
\normalfont
\label{Example 3}
{\bf Gaussian functional data model.} Let $E=C([0,1]^d), d\geq 1$ be equipped with 
the sup-norm $\|\cdot \|_{\infty}.$ Suppose that $\xi:= \sigma Z,$ where $\sigma>0$ is a known 
parameter and $Z$ is a mean zero Gaussian process on $[0,1]^d$
with the sample paths continuous a.s. (and with known distribution). 
Without loss of generality, 
assume that $\sup_{t\in [0,1]^d} {\mathbb E}Z^2(t)=1.$  
Suppose that, for some $\beta>0,$ 
$$
\tau^2(t,s):={\mathbb E}|Z(t)-Z(s)|^2 \lesssim |t-s|^{\beta},\ t,s\in [0,1]^d.
$$
Then, it is easy to see that the following bound holds for the metric entropy $H_{\tau}([0,1]^d;\eps)$
of $[0,1]^d$ with respect to metric $\tau:$
$$
H_{\tau}([0,1]^d;\eps)\lesssim_{\beta} d \log \frac{1}{\eps}.
$$
It follows from Dudley's entropy bound that
$$
{\mathbb E}\|Z\|_{\infty}^2\lesssim_{\beta} \Bigl(\int_0^1H_{\tau}^{1/2}([0,1]^d;\eps)d\eps\Bigr)^2 
\lesssim  d.
$$
Therefore, it is easy to conclude that 
$\|\Sigma\|\asymp \sigma^2$ and ${\mathbb E}\|\xi\|_{\infty}^2 \lesssim \sigma^2 d,$
implying that ${\bf r}(\Sigma)\lesssim d.$
\end{example}

In the following sections, we develop estimators $T(X)$ of $f(\theta)$ in Gaussian shift model 
with mean squared error of the order 
$$
\sup_{\|\theta\|\leq 1}{\mathbb E}_{\theta}(T(X)-f(\theta))^2 \lesssim \Bigl(\|\Sigma\|\vee ({\mathbb E}\|\xi\|^2)^s\Bigr) \wedge 1,
$$
where $s$ is the degree of smoothness of functional $f.$ We also show that this error rate 
is minimax optimal (at least in the case of standard Gaussian shift model).
Moreover, we determine a sharp threshold on smoothness $s$ such that, for all $s$ above this threshold
and all functionals $f$ of smoothness $s,$
the mean squared error rate of estimation of $f(\theta)$ is of the order $\|\Sigma\|\wedge 1$ 
(as for linear functionals), and, for all $s$ strictly above the threshold,  
we prove the efficiency of our estimators in the ``small noise" case 
(when the strong variance ${\mathbb E}\|\xi\|^2$ is small).  
The key ingredient in the development of such 
estimators is a \it bootstrap chain bias reduction \rm method introduced in \cite{Koltchinskii_2017} 
in the problem of estimation of smooth functionals of covariance operators. We will outline this approach 
in Section \ref{Overview} and develop it in detail in Section \ref{Bias_Reduction}
for Gaussian shift models. 


\section{Overview of Main Results}
\label{Overview}

We will study how the optimal error rate of estimation of $f(\theta)$ for parameter $\theta$ of 
Gaussian shift model \eqref{Gauss_shift} depends on the smoothness of the functional $f:E\mapsto {\mathbb R}$ 
as well 
as on the weak and strong variances, $\|\Sigma\|$ and ${\mathbb E}_{\Sigma}\|\xi\|^2,$ of the noise $\xi$ (or, equivalently, on the parameters $\|\Sigma\|$ and ${\bf r}(\Sigma)$). 
To this end, we define below a Banach space $C^{s,\gamma}(E)$
of functionals $f:E\mapsto {\mathbb R}$ of smoothness $s>0$ such that $f$ and its derivatives 
grow as $\|\theta\|\to \infty$ not faster than $\|\theta\|^{\gamma}$ for some $\gamma\geq 0.$

\subsection{Differentiability}

For Banach spaces $E,F,$ let ${\mathcal M}_k (E;F)$ be the Banach space of symmetric $k$-linear
forms $M:E\times \dots \times E\mapsto F$ with bounded operator norm 
$$
\|M\|:= \sup_{\|h_1\|\leq 1,\dots, \|h_k\|\leq 1}\|M(h_1,\dots, h_k)\|<\infty.
$$
For $k=0,$ ${\mathcal M}_0(E;F)$ is the space of constants
(vectors of $F$).
A function $P:E\mapsto F$ defined by $P(x):=M(x,\dots, x), x\in E,$ where 
$M\in {\mathcal M}_k (E;F),$ is called {\it a bounded homogeneous $k$-polynomial} 
on $E$ with values in $F.$ It is known that $P$ uniquely defines $M\in {\mathcal M}_k (E;F).$
{\it A bounded polynomial} on $E$ with values in $F$ is an arbitrary function $P:E\mapsto F$ represented as 
a finite sum $P(x):=\sum_{j\in I} P_j(x), x\in E, I\subset {\mathbb Z}_+,$ where $P_j$ is a non-zero bounded homogeneous $j$-polynomial. For $I=\emptyset,$ we set $P:=0.$
Polynomials $P_j, j\in I$ are uniquely defined by $P.$ The degree of $P$
is defined as ${\rm deg}(P):=\max (I)$ (with ${\rm deg}(0)=0$). If $P_j(x)=M_j(x,\dots, x)$
for $M_j\in {\mathcal M}_j (E;F),$ define 
$$
\|P\|_{\rm op}:= \sum_{j\in I}\|M_j\|.
$$ 

Recall that a function $f:E\mapsto F$ is called Fr\'echet differentiable 
at a point $x\in E$  iff there exists a bounded linear operator $f^{\prime}(x)$
from $E$ to $F$ (Fr\'echet derivative) such that 
$$
f(x+h)-f(x)= f^{\prime}(x)h + o(\|h\|)\ {\rm as}\ h\to 0.
$$
Higher order Fr\'echet derivatives could be defined by induction. 
The $k$-th order Fr\'echet derivative $f^{(k)}(x)$ at point $x$ is defined 
as the Fr\'echet derivative of the mapping $E\ni x\mapsto f^{(k-1)}(x)\in {\mathcal M}_{k-1}(E;F)$
(assuming its Fr\'echet differentiability). It is a bounded linear operator from $E$ to 
${\mathcal M}_{k-1}(E;F)$ that could be also viewed as a bounded symmetric $k$-linear
form from the space ${\mathcal M}_k(E;F).$ As always, we call $f$ $k$-times (Fr\'echet) 
continuously differentiable if its $k$-th order derivative exists and it is a continuous 
function on $E.$ Clearly, polynomials are $k$ times Fr\'echet differentiable 
for any $k.$ If $P$ is a polynomial and ${\rm deg}(P)=k,$ then $P^{(k)}$ 
is a constant (a $k$-linear symmetric form that does not depend on $x$)
and $P^{(k+1)}=0.$

We will be interested in what follows in classes of smooth functionals $f:E\mapsto {\mathbb R}$
with at most polynomial (with respect to $\|x\|$) growth of their derivatives.  To this end, we describe  
below several useful norms.

First, let $g:E\mapsto F.$ For $\gamma\geq 0,$ let 
$$
\|g\|_{L_{\infty,\gamma}}:= \sup_{x\in E}\frac{\|g(x)\|}{(1\vee \|x\|)^{\gamma}}
$$
and for $\gamma\geq 0, \rho \in (0,1],$ let 
$$
\|g\|_{{\rm Lip}_{\rho,\gamma}}:= \sup_{x'\neq x''}\frac{\|g(x')-g(x'')\|}{(1\vee \|x'\|\vee \|x''\|)^{\gamma}\|x'-x''\|^{\rho}}. 
$$
Assuming that spaces $E,F$ are equipped with their Borel $\sigma$-algebras,
we define 
$L_{\infty,\gamma}(E;F)$ as the space 
of measurable functions $g:E\mapsto F$ with 
$\|g\|_{L_{\infty,\gamma}}<\infty.$
We also define 
$$
{\rm Lip}_{\rho,\gamma}(E;F):=\{g: \|g\|_{{\rm Lip}_{\rho,\gamma}}<\infty\}.
$$ 
In the case of $F={\mathbb R},$ we will write simply $L_{\infty,\gamma}(E)$ and 
${\rm Lip}_{\rho,\gamma}(E);$ for $\gamma=0,$ we write $L_{\infty}, {\rm Lip}_{\rho}$
instead of $L_{\infty,0}, {\rm Lip}_{\rho,0}.$ 

For $k\geq 0,$ we will define the norm 
$$
\|g\|_{C^{k,\gamma}} := \max_{0\leq j\leq k}\|g^{(j)}\|_{L_{\infty,\gamma}}
$$
and the space $C^{k,\gamma}(E;F):=\{g: \|g\|_{C^{k,\gamma}}<\infty\}$
of $k$ times differentiable functions (with the growth rate of derivatives characterized 
by $\gamma$). Finally, for $s=k+\rho$ with $k\geq 0$ and $\rho \in (0,1),$ define 
$$
\|g\|_{C^{s,\gamma}} := \max_{0\leq j\leq k}\|g^{(j)}\|_{L_{\infty,\gamma}}\vee 
\|g^{(k)}\|_{{\rm Lip}_{\rho,\gamma}}
$$
and the space $C^{s,\gamma}(E;F):=\{g: \|g\|_{C^{s,\gamma}}<\infty\}.$
As before, we set $C^s :=C^{s,0}.$
It is easy to see that for any polynomial $P$ such that ${\rm deg}(P)=k$
and for all $s> 0,$ $P\in C^{s,k}(E).$

In what follows, we frequently use bounds on the remainder of the first order Taylor 
expansion 
$$
S_g(x;h):= g(x+h)-g(x)-g^{\prime}(x)(h), x,h\in E
$$
of Fr\'echet differentiable function $g:E\mapsto {\mathbb R}.$
We will skip the proof of the following simple lemma.

\begin{lemma}
\label{Taylor_remainder}
Assume that $g:E\mapsto {\mathbb R}$ is Fr\'echet differentiable in $E$ with 
$g^{\prime}\in {\rm Lip}_{\rho,\gamma}(E;{\mathcal M}_1(E;F)).$ 
Then 
$$
|S_g(x;h)| \lesssim \|g^{\prime}\|_{{\rm Lip}_{\rho,\gamma}} (1\vee \|x\|\vee \|h\|)^{\gamma}\|h\|^{1+\rho}, x,h\in E
$$
and 
$$
|S_g(x;h')-S_g(x;h)|\lesssim \|g^{\prime}\|_{{\rm Lip}_{\rho,\gamma}} (1\vee \|x\|\vee \|h\|\vee \|h'\|)^{\gamma} (\|h\|\vee \|h'\|)^{\rho}\|h'-h\|,\ x,h,h'\in E.
$$
\end{lemma}

\subsection{Definition of estimators and risk bounds}

The crucial step in construction of estimator $T_k$ is a bias reduction method developed in detail in Section \ref{Bias_Reduction}
and briefly outlined here. Consider the following linear operator 
$$
\T g(\theta)\coloneqq \E_{\theta}g(X)= \E g(\theta +\xi), \theta \in E
$$
that is well defined on the spaces $L_{\infty, \gamma}(E)$ for $\gamma\geq 0.$  
Given a smooth functional $f:E\mapsto {\mathbb R},$ we would like to find a functional $g$
on $E$ such that the bias of estimator $g(X)$ of $f(\theta)$ is small enough. In other words, we would 
like to find an approximate solution of operator equation $\T g(\theta)=f(\theta), \theta \in E.$ 
Under the assumption that the strong variance ${\mathbb E}\|\xi\|^2$ of the noise $\xi$ is small,
the operator $\T$ is close to the identity operator $\I.$ Define $\B := \T-\I.$ Then, at least formally,
the solution of the equation $\T g(\theta)=f(\theta), \theta\in E$ could be written as a Neumann series:
$$
g= (\I+\B)^{-1} f = (\I-\B+\B^2-\B^3+\dots)f.
$$
We will define an estimator $f_k(X)$ in terms of a partial sum of this series:
$$
f_k(\theta):= \sum_{j=0}^k (-1)^{j} \B^j f(\theta), \theta\in E.
$$
It will be proved in Section \ref{Bias_Reduction}, that, for this estimator, the bias ${\mathbb E}_{\theta} f_k(X)-f(\theta)$
is of the order $\lesssim ({\mathbb E}^{1/2}\|\xi\|^2)^s,$ provided that $f\in C^{s,\gamma}(E)$ for $s=k+1+\rho,$ 
$k\geq 0,$ $\rho\in (0,1]$ and $\|\theta\|$ is bounded by a constant.

We will prove in Section \ref{sec.conc} the following result.

\begin{theorem}
\label{first_result}
Let $s=k+1+\rho$ for some $\rho\in (0,1]$ and let $\gamma\geq 0.$
Suppose that $f\in C^{s,\gamma}(E).$ 
Let 
\begin{equation}
\nonumber
T_k(X):=
\begin{cases}
f_k(X)\ {\rm if}\ {\mathbb E}^{1/2}\|\xi\|^2\leq 1/2\\
0\ \ \ \ \ \ \  {\rm otherwise.}
\end{cases}
\end{equation}
Then
\begin{align}
\label{upper_mean_square}
{\mathbb E}_{\theta}(T_k(X)-f(\theta))^2\lesssim_{\gamma} (k+1)^{\gamma} \|f\|_{C^{s, \gamma}}^2
(1\vee \|\theta\|)^{2\gamma}\Bigl(\Bigl(\|\Sigma\| \vee ({\mathbb E}\|\xi\|^2)^s\Bigr)\wedge 1\Bigr).
\end{align}
\end{theorem}

It follows from bound \eqref{upper_mean_square} that 
\begin{align}
\label{upper_mean_square_A}
\sup_{\|f\|_{C^{s,\gamma}}\leq 1} \sup_{\|\theta\|\leq 1} 
{\mathbb E}_{\theta}(T_k(X)-f(\theta))^2 \lesssim_{s,\gamma} 
\Bigl(\Bigl(\|\Sigma\|\vee ({\mathbb E}\|\xi\|^2)^s\Bigr)\wedge 1\Bigr).
\end{align}
We will show in Section \ref{nemirovski} that, {\it in the case of standard Gaussian shift model}, the above bound is optimal in a minimax sense. More precisely, in this case, the following result holds.

\begin{theorem}
\label{min_max_nemirovski}
Let $E:={\mathbb R}^d$ (equipped with the standard Euclidean norm) and let 
$X\sim {\mathcal N}(\theta; \sigma^2 I_d), \theta \in {\mathbb R}^d$ for some $\sigma^2>0.$
Then 
\begin{align}
\label{min_max_mean_square}
\sup_{\|f\|_{C^s}\leq 1} \inf_{T} \sup_{\|\theta\|\leq 1} {\mathbb E}_{\theta}(T(X)-f(\theta))^2 
& \gtrsim 
\Bigl(\|\Sigma\|\vee \Bigl({\mathbb E}\|\xi\|^2\Bigr)^{s}\Bigr) \wedge 1,
\end{align}
where the infimum is taken over all possible estimators $T(X).$
\end{theorem}

At this point, we could not extend 
the lower bound of Theorem \ref{min_max_nemirovski} to general Gaussian shift models in Banach spaces. 

\begin{remark}
\normalfont 
In a very recent paper \cite{Fan_Zhou}, Zhou and Li state a similar result (Theorem 7.2) with 
Besov $B^s_{\infty,1}$-norm instead of $C^s$-norm. There is a mistake in the proof of this 
result (contrary to the claim of the authors, it is impossible to choose function $\varphi$ used 
in the proof so that $\|\tilde \varphi\|_{B^s_{\infty,1}}\leq 1$ and other required properties hold). 
However, if the Besov norm is replaced by $C^s$-norms used in our paper, their proof seems correct.
The method of the proof of Theorem 7.2 in \cite{Fan_Zhou} differs from ours.
\end{remark}

\subsection{Efficiency}

Bound \eqref{upper_mean_square_A} implies that, if the smoothness $s$ of functional $f$ is sufficiently 
large, namely if 
\begin{equation}
\label{critical_smooth}
({\mathbb E}\|\xi\|^2)^s\leq \|\Sigma\|,
\end{equation}
then 
\begin{align}
\label{upper_mean_square_B}
\sup_{\|f\|_{C^{s,\gamma}}\leq 1} \sup_{\|\theta\|\leq 1} 
{\mathbb E}_{\theta}(T_k(X)-f(\theta))^2 \lesssim_{s,\gamma} 
\|\Sigma\|\wedge 1,
\end{align}
which coincides with the largest minimax optimal mean squared error 
for linear functionals from the unit ball in $E^{\ast}.$ 
Assuming that ${\mathbb E}\|\xi\|^2\leq 1,$ condition \eqref{critical_smooth} can be equivalently written 
as 
\begin{equation}
\label{critical_smooth''}
s\geq 1+ \frac{\log {\bf r}(\Sigma)}{\log \frac{1}{\|\Sigma\|}-\log {\bf r}(\Sigma)}.
\end{equation}
If $\sigma^2 :=\|\Sigma\|$ is a small parameter and ${\bf r}(\Sigma)\leq \sigma^{-2\alpha}$
for some $\alpha \in (0,1),$ condition \eqref{critical_smooth''} would follow from 
the condition $s\geq \frac{1}{1-\alpha}.$ On the other hand, it follows from bound \eqref{min_max_mean_square}
that, in the case of standard Gaussian shift model, the smoothness threshold $\frac{1}{1-\alpha}$ is sharp
for estimation with mean squared error rate $\asymp \sigma^2.$ Indeed, in this case, ${\bf r}(\Sigma)=d$
and, if $\sigma$ is small and $d\asymp \sigma^{-2\alpha}$ for some $\alpha\in (0,1),$ then, for any $s<\frac{1}{1-\alpha},$ there exists a functional $f$ with $\|f\|_{C^{s,\gamma}}\leq 1$ such that  
\begin{align*}
\inf_{T}\sup_{\|\theta\|\leq 1} 
{\mathbb E}_{\theta}(T(X)-f(\theta))^2 \gtrsim 
\sigma^{2s(1-\alpha)},
\end{align*}
which is significantly larger than $\sigma^2$ as $\sigma \to 0.$ 
Moreover, if $d\asymp \sigma^{-2},$ then, for any $s>0,$ there exists a functional $f$ with $\|f\|_{C^{s,\gamma}}\leq 1$ such that  
\begin{align*}
\inf_{T}\sup_{\|\theta\|\leq 1} 
{\mathbb E}_{\theta}(T(X)-f(\theta))^2 \gtrsim_s 
1,
\end{align*}
essentially implying that even consistent estimators of $f(\theta)$ do not exist in this case.

In the case when ${\bf r}(\Sigma)\lesssim \sigma^{-2\alpha}$ for some $\alpha\in (0,1)$ and $s>\frac{1}{1-\alpha}$ (or, more generally, when $({\mathbb E}\|\xi\|^2)^s$ is of a smaller order than $\|\Sigma\|$), it is possible to prove that $f_k(X)-f(\theta)$ is close in distribution to normal 
and establish the efficiency of estimator $f_k(X).$ 
More precisely, 
let 
$$
\sigma_{f,\xi}^2 (\theta):= {\mathbb E}(f'(\theta)(\xi))^2 
=\langle \Sigma f'(\theta), f'(\theta)\rangle
$$
For $s\geq 1, \gamma\geq 0,$ denote
$$
K(f;\Sigma;\theta):=K_{s,\gamma}(f;\Sigma;\theta):=
\frac{\|f\|_{C^{s,\gamma}}(1\vee \|\theta\|)^{\gamma}\|\Sigma\|^{1/2}}{\sigma_{f,\xi}(\theta)}.
$$
It is easy to see that 
$$
\sigma_{f,\xi}(\theta)\leq \|\Sigma\|^{1/2}\|f^{\prime}(\theta)\|
\leq \|f^{\prime}\|_{L_{\infty,\gamma}}(1\vee \|\theta\|)^{\gamma}\|\Sigma\|^{1/2}
\leq 
\|f\|_{C^{s,\gamma}}(1\vee \|\theta\|)^{\gamma}\|\Sigma\|^{1/2},
$$
implying that $K_{s,\gamma}(f;\Sigma;\theta)\geq 1.$ We also have that 
$$
K_{s,\gamma}(f;\lambda \Sigma;\theta)=K_{s,\gamma}(f;\Sigma;\theta), \lambda>0,
$$
which means that $K_{s,\gamma}(f;\Sigma;\theta)$ does not depend on the noise level
$\|\Sigma\|^{1/2}.$ In what follows, it will be assumed that the functional 
$K_{s,\gamma}(f;\Sigma;\theta)$ is bounded from above by a constant,
implying that $\sigma_{f,\xi}(\theta)$ is within a constant from its upper 
bound $\|f\|_{C^{s,\gamma}}(1\vee \|\theta\|)^{\gamma}\|\Sigma\|^{1/2}.$
This is the case, for instance, when $\theta$ is in a bounded set and $\sigma_{f,\xi}(\theta)\gtrsim \|\Sigma\|^{1/2}$
(in other words, the standard deviation $\sigma_{f,\xi}(\theta)$ is not too small 
comparing with the noise level $\|\Sigma\|^{1/2}$).

The following result will be proved in Section \ref{sec:norm_appr}.

\begin{theorem}
\label{norm_appr_th_1}
Suppose, for some $s=k+1+\rho,$ $\rho\in (0,1]$ and some $\gamma>0,$ $f\in C^{s,\gamma}(E).$
Suppose also that  ${\mathbb E}^{1/2}\|\xi\|^2\leq 1/2.$ 
Then 
\begin{align}
\label{norm_appr}
&
\nonumber
\sup_{y\in {\mathbb R}}
\biggl |
{\mathbb P}_{\theta}
\biggl\{
\frac{f_k(X) - f(\theta)}{\sigma_{f,\xi}(\theta)}\leq y
\biggr\}
-{\mathbb P}\{Z\leq y\}
\biggr |
\lesssim_{\gamma} 
(k+1)^{\gamma/2} 
K_{s,\gamma}(f;\Sigma;\theta)
\\
&
\biggl(
({\mathbb E}^{1/2}\|\xi\|^2)^{\rho}\sqrt{\log\biggl(\frac{1}{\|\Sigma\|}\biggr)}
\bigvee \|\Sigma\|^{\rho/2}\log^{(1+\rho)/2}\biggl(\frac{1}{\|\Sigma\|}\biggr)
\bigvee \frac{({\mathbb E}^{1/2}\|\xi\|^2)^{s}}{\|\Sigma\|^{1/2}}
\biggr),
\end{align}
where $Z$ is a standard normal r.v.
Moreover,
\begin{align}
\label{L_2_remainder}
&
\nonumber
\biggl\|\frac{f_k(X) - f(\theta)}{\sigma_{f,\xi}(\theta)}-Z\biggr\|_{L_2({\mathbb P})}
\\
&
\lesssim_{\gamma}
(k+1)^{\gamma/2} K_{s,\gamma}(f;\Sigma;\theta)
\biggl(({\mathbb E}^{1/2}\|\xi\|^2)^{\rho}
\bigvee \frac{({\mathbb E}^{1/2}\|\xi\|^2)^s}{\|\Sigma\|^{1/2}}
\biggr).
\end{align}
\end{theorem}

It follows from bound \eqref{L_2_remainder} that 
\begin{align}
\label{L_2_up}
&
\nonumber
\frac{{\mathbb E}_{\theta}^{1/2}(f_k(X) - f(\theta))^2}{\sigma_{f,\xi}(\theta)}
\\
&
\leq 
1+ c_{\gamma}(k+1)^{\gamma/2} K_{s,\gamma}(f;\Sigma;\theta)
\biggl(({\mathbb E}^{1/2}\|\xi\|^2)^{\rho}
\bigvee \frac{({\mathbb E}^{1/2}\|\xi\|^2)^s}{\|\Sigma\|^{1/2}}
\biggr).
\end{align}
Assume that $\theta$ is in a set $\Theta\subset E$ of parameters where $K_{s,\gamma}(f;\Sigma;\theta)$ is upper bounded by a constant. 
Then, $\frac{{\mathbb E}_{\theta}^{1/2}(f_k(X) - f(\theta))^2}{\sigma_{f,\xi}(\theta)}$ is close to $1$ uniformly in $\Theta$ provided that 
${\mathbb E}\|\xi\|^2$ is small and $({\mathbb E}\|\xi\|^2)^s$ is 
much smaller than $\|\Sigma\|$ (say, if ${\bf r}(\Sigma)\lesssim \sigma^{-2\alpha}$ and $s>\frac{1}{1-\alpha}$).

Finally, in Section \ref{Sec:Lowerbounds}, we will prove the following minimax lower bound. 

\begin{theorem}
\label{min_lower_bd}
Suppose $f\in C^{s,\gamma}(E)$ for some $s\in (1,2]$ and $\gamma\geq 0.$
Let 
$\theta_0\in \overline{{\rm Im}(\Sigma)}.$ 
Then, there exists a constant $D_{\gamma}>0$ such that 
for all $c>0$ and all covariance operators $\Sigma$ satisfying the condition 
$
c\|\Sigma\|^{1/2}\leq 1,
$ 
the following bound holds  
\begin{align*}
&
\inf_{T}\sup_{\|\theta-\theta_0\|\leq c\|\Sigma\|^{1/2}}
\frac{\E_{\theta} (T(X)-f(\theta))^2}{\sigma^{2}_{f,\xi}(\theta)} 
\geq 1-D_{\gamma} K_{s,\gamma}^{2}(f;\Sigma;\theta_0) 
\Bigl(c^{s-1}\|\Sigma\|^{(s-1)/2}+\frac{1}{c^2}\Bigr),
\end{align*}
where the infimum is taken over all possible estimators $T(X).$
\end{theorem}

The bound of Theorem \ref{min_lower_bd} shows that, when the noise level 
$\|\Sigma\|^{1/2}$ is small and $K_{s,\gamma}(f;\Sigma;\theta_0)$ is upper 
bounded by a constant, the following asymptotic minimax result (in spirit 
of H\`ajek and Le Cam) holds
$$
\lim_{c\to\infty}\liminf_{\|\Sigma\|^{1/2}\to 0}\inf_{T}\sup_{\|\theta-\theta_0\|\leq c\|\Sigma\|^{1/2}}
\frac{\E_{\theta} (T(X)-f(\theta))^2}{\sigma^{2}_{f,\xi}(\theta)} \geq 1
$$ 
locally in a neighborhood of parameter $\theta_0$ of size commensurate with the noise level.
This shows the optimality of the variance $\sigma_{f,\xi}^2(\theta)$ of normal approximation and the efficiency of estimator $f_k(X).$

\begin{remark}
\normalfont
In the case of matrix Gaussian shift model of Example \ref{Example 2} (that is, when $E$ is the space 
of symmetric $d\times d$ matrices equipped with operator norm and $\xi=\sigma Z,$ $Z$ being 
a random matrix from Gaussian orthogonal ensemble), 
the results of the paper 
could be applied, in particular, to bilinear forms of smooth functions of $d\times d$ symmetric matrices:
$f(\theta):=\langle h(\theta) u,v\rangle,$ where $h$ is a smooth function in real line and  
$u,v\in {\mathbb R}^d.$ Namely, it was shown in \cite{Koltchinskii_2017}, Corollary 2 (based on the results of \cite{Peller_87}, \cite{Aleksandrov_Peller}) that the $C^s$-norm of operator function $\theta\mapsto h(\theta)$ can be controlled in terms of Besov $B^s_{\infty,1}$-norm of underlying function of real variable $h:$ $\|h\|_{C^s(E)}\lesssim_s \|h\|_{B_{\infty,1}^s}, s>0.$ This allows one to apply 
all the results stated above to functional $f(\theta)$ provided that $h$ is in a proper Besov space. 
Note that spectral projections of $\theta$ that correspond to subsets of its spectrum separated by 
a positive gap from the rest of the spectrum could be represented as $h(\theta)$ for sufficiently smooth
functions $h,$ which allows one to apply the results to bilinear forms of spectral projections
(see also \cite{Koltchinskii_Xia}). 
In \cite{Koltchinskii_2017}, similar results were obtained for smooth functionals of covariance 
operators. 
\end{remark}

\begin{remark}
\normalfont 
Obviously, the results of the paper can be applied to the model of i.i.d. observations $X_1,\dots, X_n\sim 
{\mathcal N}(\theta;\Sigma),\ \theta \in E.$ If $\bar X:= \frac{X_1+\dots +X_n}{n},$ then it follows 
from Theorem \ref{first_result} that 
\begin{align}
\label{upper_mean_square_barX}
{\mathbb E}_{\theta}(T_k(\bar X)-f(\theta))^2\lesssim_{\gamma} (k+1)^{\gamma} \|f\|_{C^{s, \gamma}}^2
(1\vee \|\theta\|)^{2\gamma}\Bigl(\Bigl(\frac{\|\Sigma\|}{n} \bigvee 
\biggl(\frac{\|\Sigma\|{\bf r}(\Sigma)}{n}\biggr)^s\bigwedge 1\Bigr).
\end{align}
Uniformly in the class of covariances with $\|\Sigma\|\lesssim 1$ and ${\bf r}(\Sigma)\lesssim n^{\alpha}$ for some $\alpha\in (0,1),$
this yields a bound on the mean squared error of the order $O(\frac{1}{n})$ provided that $s\geq \frac{1}{1-\alpha}.$
Moreover, if $s>\frac{1}{1-\alpha},$ estimator $f_k(\bar X)$ is asymptotically normal and asymptotically efficient with convergence rate $\sqrt{n}$
and limit variance $\sigma_{f,\xi}^2(\theta).$
\end{remark}

\section{Bias Reduction}
\label{Bias_Reduction}

A crucial part of our approach to efficient estimation of smooth functionals of $\theta$
is a new bias reduction method based on iterative application of parametric bootstrap. 
Our goal is to construct an estimator of smooth functional $f(\theta)$ of parameter $\theta\in E$
and, to this end, we construct an estimator of the form $g(X)$ for some functional $g:E\mapsto {\mathbb R}$
for which the bias $\E_{\theta} g(X)-f(\theta)$ is negligible comparing with the noise level 
$\|\Sigma\|^{1/2}.$ Define the following linear operator: 
$$
\T g(\theta)\coloneqq \E_{\theta}g(X)= \E g(\theta +\xi), \theta \in E.
$$

\begin{proposition}
For all $\gamma\geq 0,$ $\T$ is a bounded linear operator 
from the space $L_{\infty,\gamma}(E)$ into itself with 
\begin{equation}
\label{bd_norm_T}
\|\T\|_{L_{\infty,\gamma}(E)\mapsto L_{\infty,\gamma}(E)} \leq 
2^{\gamma}(1+\E\|\xi\|^{\gamma}).
\end{equation}
\end{proposition}

\begin{proof}
Indeed, by the definition of $L_{\infty,\gamma}$-norm, 
$$
|g(\theta +\xi)|\leq 2^{\gamma}\|g\|_{L_{\infty, \gamma}} (1\vee \|\theta\|\vee \|\xi\|)^{\gamma}.
$$
Therefore,
$$
|\T g(\theta)| \leq \E|g(\theta+\xi)| \leq 2^{\gamma}\|g\|_{L_{\infty, \gamma}}  
\E(1\vee \|\theta\|\vee \|\xi\|)^{\gamma}\leq 
2^{\gamma}
[(1\vee \|\theta\|)^{\gamma}+\E\|\xi\|^{\gamma}]
\|g\|_{L_{\infty, \gamma}}, 
$$
which easily implies that 
\begin{equation}
\label{bd_T}
\|\T g\|_{L_{\infty,\gamma}}\leq  2^{\gamma}(1+\E\|\xi\|^{\gamma})
\|g\|_{L_{\infty,\gamma}}.
\end{equation}
Therefore $\T$ is a bounded operator from $L_{\infty,\gamma}(E)$ 
into itself and bound \eqref{bd_norm_T} holds.
\qed
\end{proof}

The following proposition could be easily proved by induction.

\begin{proposition}
\label{T^k}
Let $\xi_1,\xi_2, \dots$ be i.i.d. copies of $\xi$ and let $g\in L_{\infty,\gamma}(E)$
for some $\gamma>0.$ Then, for all $k\geq 1,$
$$
\T^k g(\theta)= \E g\biggl(\theta+\sum_{j=1}^k \xi_j\biggr),
\theta \in E.
$$
\end{proposition}

Note that, by a simple modification of the proof of bound \eqref{bd_T}, 
we can derive from Proposition \ref{T^k} that 
\begin{equation}
\label{bd_T^k}
\|\T^kg\|_{L_{\infty,\gamma}}\leq  2^{\gamma}(1+k^{\gamma/2}\E\|\xi\|^{\gamma})\|g\|_{L_{\infty,\gamma}}.
\end{equation}

To find an estimator $g(X)$ of $f(\theta)$ with a small bias it suffices to solve (approximately)
the equation $(\T g)(\theta)=f(\theta), \theta \in E.$ Denote 
$\B=:\T-\I .$ For a small level of noise $\xi,$
one can expect operator $\B$ to be ``small". The solution of 
equation $\T g=f$ could be then formally written as a Neumann 
series:
$$
g= (\I +\B)^{-1}f = (\I -\B+\B^2-\dots)f.
$$
We use a partial sum of this series as an approximate solution
$$
f_k(\theta):= \sum_{j=0}^k (-1)^j \B^j f(\theta), \theta \in E
$$
and consider in what follows the estimator $f_k(X)$ of $f(\theta).$

Our main goal in this section is to prove the following theorem that provides 
an upper bound on the bias of estimator $f_k(X).$

\begin{theorem}
\label{bias_control}
Let $s=k+1+\rho$ for some $\rho\in (0,1]$ and let $\gamma\geq 0.$
Suppose that $f\in C^{s,\gamma}(E).$
Denote by ${\frak B} f_k(\theta):= {\mathbb E }_{\theta}f_k(X)-f(\theta), \theta\in E$ 
the bias of estimator $f_k(X).$
Then 
\begin{align*}
\|{\frak B}f_k\|_{L_{\infty,\gamma}}
\lesssim 
2^{\gamma}\|f^{(k+1)}\|_{{\rm Lip}_{\rho,\gamma}} 
(1+ k^{\gamma/2}\E^{1/2}\|\xi\|^{2\gamma})
(1+\E^{1/2}\|\xi\|^{2\gamma})
(\E^{1/2}\|\xi\|^2)^{s}. 
\end{align*} 
\end{theorem}

By a straightforward simple computation, the bias of 
estimator $f_k(X)$ is equal to 
\begin{equation}
\label{bias_111}
{\frak B}f_k(\theta)=\E f_k(X)-f(\theta)= (-1)^k \B^{k+1}f(\theta).
\end{equation}
This leaves us with the problem of bounding $\B^{k+1}f(\theta)$ 
for a sufficiently smooth function $f.$ By Newton's Binomial Formula,
for all $k\geq 1,$
\begin{equation}
\label{Newton}
\B^k f(\theta)= \sum_{j=0}^k (-1)^{k-j}{k\choose j} T^{j}f(\theta), \theta \in E.
\end{equation}
It follows from representation \eqref{Newton} and bound \eqref{bd_T^k} that 
\begin{equation}
\label{bd_B^k}
\|\B^kg\|_{L_{\infty,\gamma}}\leq  2^{\gamma}
\sum_{j=0}^k {k\choose j}(1+j^{\gamma/2}\E\|\xi\|^{\gamma})\|g\|_{L_{\infty,\gamma}}
\leq 
2^{k+\gamma}\|g\|_{L_{\infty,\gamma}}(1+k^{\gamma/2}\E\|\xi\|^{\gamma}).
\end{equation}

\begin{remark}
\normalfont
Define $\hat \theta^{(k)}:= \theta + \sum_{j=1}^k \xi_j, k\geq 1$ and $\hat \theta^{(0)}:=\theta.$
Then $\hat \theta^{(1)}=\hat \theta =X$ is the maximum likelihood estimator of parameter $\theta,$ 
$\hat \theta^{(2)}$ is a parametric bootstrap estimator corresponding to $\hat \theta,$
and 
$\hat \theta^{(k)}, k\geq 2$ could be viewed as successive iterations of parametric bootstrap for Gaussian shift model $X\sim {\mathcal N}(\theta, \Sigma), \theta \in E.$ Similar sequence of bootstrap estimators (that form a Markov chain)
was studied in \cite{Koltchinskii_2017} in the case of covariance estimation and it was called {\it a bootstrap chain.}
It immediately follows from \eqref{Newton} and Proposition \ref{T^k} that 
\begin{equation}
\label{finite_diff_AAA}
\B^k f(\theta)= {\mathbb E}_{\theta}\sum_{j=0}^k (-1)^{k-j}{k\choose j} f(\hat \theta^{(j)}), \theta \in E,
\end{equation}
which means that $\B^k f(\theta)$ is equal to the expectation of the $k$-th order difference of sequence $f(\hat \theta^{j}), j\geq 0.$ The bias reduction method studied in this section is a special 
case of a more general bootstrap chain bias reduction developed in the case of estimation 
of functionals of covariance in \cite{Koltchinskii_2017}. Operators similar to $\B^k$ were used 
also in \cite{Jiao} in the problem of bias reduction in estimation of $f(\theta),$ where $\theta$
is the parameter of binomial model. In this case, $\T$ maps function $f$ to the corresponding 
Bernstein polynomial and bounds on $\B^k f(\theta)$ could be obtained using some results in approximation theory.
\end{remark}

For sufficiently smooth functions $f,$ we will derive a more convenient 
integral representation of functions $\B^k f$ that would yield 
sharper bounds on their $L_{\infty,\gamma}$ norms.

\begin{theorem}
\label{repr_B^k}
Suppose $f\in C^{k,\gamma}(E)$ for some $\gamma\geq 0.$
Then 
$$
\B^k f(\theta)= \E f^{(k)}\biggl(\theta+ \sum_{j=1}^k \tau_j \xi_j\biggr)(\xi_1,\dots, \xi_k), \theta\in E,
$$  
where $\tau_{1},\dots,\tau_{k}\sim U[0,1]$ are i.i.d. random variables independent of $\xi_{1},\dots,\xi_{k}.$
\end{theorem}

\begin{proof}
Define 
$$
\varphi(t_1,\dots, t_k):= f\biggl(\theta+ \sum_{i=1}^k t_i \xi_i\biggr), (t_1,\dots, t_k)\in [0,1]^k.
$$
It immediately follows from Proposition \ref{T^k} that
$$
\T^j f(\theta) = \E\varphi(t_1,\dots, t_k)
$$
for all $j\leq k$ and for all $(t_1,\dots, t_k)\in \{0,1\}^k$ with $\sum_{i=1}^k t_i=j.$
This allows us to rewrite representation \eqref{Newton} as follows:
$$
\B^k f(\theta)= \E\sum_{(t_1,\dots, t_k)\in \{0,1\}^k} (-1)^{k-\sum_{i=1}^k t_i} \varphi (t_1,\dots, t_k).
$$
For functions $\phi :[0,1]^k\mapsto {\mathbb R},$ define the first order difference operators $\Delta^{(i)}, i=1,\dots, k:$ 
$$\Delta^{(i)}\phi(t_{1},\dots,t_{k})\coloneqq\phi(t_{1},\dots,t_{k})\big\rvert_{t_{i}=1}-\phi(t_{1},\dots,t_{k})\big\rvert_{t_{i}=0}.
$$
It is easy to show by induction that 
$$
\Delta^{(1)} \dots \Delta^{(k)} \phi = \sum_{(t_1,\dots, t_k)\in \{0,1\}^k} (-1)^{k-\sum_{i=1}^k t_i} \phi (t_1,\dots, t_k),
$$
implying that 
$$
\B^k f(\theta)= \E\Delta^{(1)} \dots \Delta^{(k)} \varphi.
$$
For $f\in C^{k,\gamma}(E),$ the function $\varphi$ is $k$ times continuously 
differentiable on $[0,1]^k$ with 
$$
\frac{\partial^k \varphi(t_1,\dots, t_k)}{\partial t_1 \dots \partial t_k}= 
f^{(k)}\biggl(\theta+ \sum_{j=1}^k t_j \xi_j\biggr)(\xi_1,\dots, \xi_k).   
$$
By generalized Newton-Leibnitz formula, 
$$
\Delta^{(1)} \dots \Delta^{(k)} \varphi = \int_{0}^1 \dots \int_0^1 \frac{\partial^k \varphi(t_1,\dots, t_k)}{\partial t_1 \dots \partial t_k} dt_1 \dots dt_k. 
$$
Therefore,
$$
\B^k f(\theta)= \E\int_{0}^1 \dots \int_0^1 
f^{(k)}\biggl(\theta+ \sum_{j=1}^k t_j \xi_j\biggr)(\xi_1,\dots, \xi_k)
dt_1 \dots dt_k,
$$
which implies the result.

\qed
\end{proof}

\begin{corollary}
\label{bias_polynomial}
Let $f:E\mapsto {\mathbb R}$ be a polynomial of degree $k+1\geq 1.$ 
Then $\B^{k+1} f =0$ and, as a consequence, $f_{k}(X)$ is an unbiased 
estimator of $f(\theta).$ Moreover, $\B^j f =0$ for all $j>(k+1)/2$ implying 
that $f_{[\frac{k+1}{2}]}(X)$ is an unbiased estimator of  $f(\theta).$
\end{corollary}

\begin{proof}
Note that $f\in C^{s,k+1}(E)$ for all $s>0.$
Since $f^{(k+1)}(\theta)=M,\theta \in E$ for some $M\in {\mathcal M}_{k+1}(E;{\mathbb R}),$
we can use independence of $\xi_1,\dots, \xi_{k+1}$ to get 
$$
\B^{k+1} f(\theta)= \E M(\xi_1,\dots, \xi_{k+1})= M(\E \xi_1,\dots, \E \xi_{k+1})=0
$$  
and 
$$
{\frak B}f_{k}(\theta) = (-1)^{k}\B^{k+1} f(\theta)=0.
$$

To prove the second claim, note that, for $j>(k+1)/2,$ $f^{(k)}$ is a polynomial of degree
$k+1-j<j.$ Using Taylor expansion, for a fixed $\theta,$ $x\mapsto f^{(j)}(\theta+x)$ 
is also a polynomial of degree $k+1-j.$ Therefore, we can now represent 
$$   
f^{(j)}\Bigl(\theta+\sum_{i=1}^j \tau_i \xi_i\Bigr)(\xi_1,\dots, \xi_j)
$$
as a sum of multilinear forms
$$
M(\xi_{i_1}, \dots, \xi_{i_l})(\xi_1,\dots, \xi_j),\ l\leq k+1-j <j, 1\leq i_1,\dots, i_l\leq j. 
$$ 
Since $\{i_1,\dots, i_l\}$ is a strict subset of $\{1,\dots, j\},$ we easily get by conditioning 
that 
$$
\E M(\tau_{i_1}\xi_{i_1}, \dots, \tau_{i_l}\xi_{i_l})(\xi_1,\dots, \xi_j)=0,
$$
which implies that 
$$
\B^{j} f(\theta)= \E f^{(j)}\Bigl(\theta+\sum_{i=1}^j \tau_i \xi_i\Bigr)(\xi_1,\dots, \xi_j)=0, \theta\in E.
$$

\qed
\end{proof}

\begin{remark}
\normalfont
Other representations of unbiased estimators of polynomials of parameter 
$\theta$ of Gaussian shift model (especially, in the case of standard model of Example \ref{Example 1})
could be found in the literature (in particular, see \cite{Ibragimov_Khasm_Nemirov}). 
\end{remark}

Representation of Theorem \ref{repr_B^k} could be now used to provide 
an upper bound on $L_{\infty,\gamma}$-norm of function $\B^k f.$

\begin{proposition}
For all $\gamma\geq 0$ and all $f\in C^{k,\gamma}(E),$ the following bound holds:
$$
\|\B^k f\|_{L_{\infty,\gamma}} \leq 2^{\gamma}
\|f^{(k)}\|_{L_{\infty, \gamma}}\Bigl(1+ k^{\gamma/2} \E^{1/2}\|\xi\|^{2\gamma}\Bigr)(\E^{1/2}\|\xi\|^2)^k.
$$
\end{proposition}

\begin{proof}
Observe that  
\begin{align*}
&
\Bigl|f^{(k)}\biggl(\theta+\sum_{j=1}^k \tau_j \xi_j\biggr)(\xi_1,\dots, \xi_k)\Bigr|
\leq \|f^{(k)}\|_{L_{\infty,\gamma}} \biggl(1\vee \biggl\|\theta+\sum_{j=1}^k \tau_j \xi_j\biggr\|\biggr)^{\gamma}
\|\xi_1\|\dots \|\xi_k\|
\\
&
\leq 2^{\gamma}\|f^{(k)}\|_{L_{\infty,\gamma}} \biggl(1\vee \|\theta\|\vee \biggl\|\sum_{j=1}^k \tau_j \xi_j\biggr\|\biggr)^{\gamma}
\|\xi_1\|\dots \|\xi_k\|,
\end{align*}
implying that 
\begin{align*}
&
|\B^k f(\theta)|
\leq 
2^{\gamma}\|f^{(k)}\|_{L_{\infty,\gamma}} 
\E^{1/2}\biggl(1\vee \|\theta\|\vee \biggl\|\sum_{j=1}^k \tau_j \xi_j\biggr\|\biggr)^{2\gamma}
\E^{1/2}(\|\xi_1\|^2\dots \|\xi_k\|^2)
\\
&
\leq 2^{\gamma}\|f^{(k)}\|_{L_{\infty,\gamma}} 
\biggl((1\vee \|\theta\|)^{2\gamma} + 
\E\biggl\|\sum_{j=1}^k \tau_j \xi_j\biggr\|^{2\gamma}\biggr)^{1/2}
(\E^{1/2}\|\xi\|^2)^k.
\end{align*}
Next note that, conditionally on $\tau_1,\dots, \tau_k,$ the distribution of $\sum_{j=1}^k \tau_j \xi_j$
is the same as the distribution of r.v. $\Bigl(\sum_{j=1}^k \tau_j^2\Bigr)^{1/2}\xi.$ Therefore,
$$
\E\biggl\|\sum_{j=1}^k \tau_j \xi_j\biggr\|^{2\gamma}= 
\E\Bigl(\sum_{j=1}^k \tau_j^2\Bigr)^{\gamma} \E\|\xi\|^{2\gamma}
\leq k^{\gamma} \E\|\xi\|^{2\gamma},
$$
and we get 
\begin{align*}
&
|\B^k f(\theta)|
\leq 
2^{\gamma}\|f^{(k)}\|_{L_{\infty,\gamma}} 
\Bigl((1\vee \|\theta\|)^{\gamma}+ k^{\gamma/2} \E^{1/2}\|\xi\|^{2\gamma}\Bigr)
(\E^{1/2}\|\xi\|^2)^k.
\end{align*}
This yields the bound of the proposition.

\qed
\end{proof}

The next corollary is immediate.

\begin{corollary}
\label{cor_f_k_infty_gamma}
Suppose $\E^{1/2}\|\xi\|^2\leq 1/2.$
For all $\gamma\geq 0$ and all $f\in C^{k,\gamma}(E),$ the following bound holds:
$$
\|f_k\|_{L_{\infty,\gamma}} \leq 2^{\gamma+1}
\|f\|_{C^{k, \gamma}}\Bigl(1+ k^{\gamma/2} \E^{1/2}\|\xi\|^{2\gamma}\Bigr)
\lesssim_{\gamma} (k+1)^{\gamma/2} \|f\|_{C^{k, \gamma}}.
$$
\end{corollary}

\begin{theorem}
\label{repr_deriv_B^k}
Suppose $f\in C^{k+1,\gamma}(E)$ for some $\gamma\geq 0.$
Then $\theta \mapsto \B^{k} f(\theta)$ is Fr\'echet 
differentiable with continuous derivative
\begin{equation}
\label{deriv_formula}
(\B^k f)'(\theta)(h)= 
\E f^{(k+1)}\biggl(\theta+ \sum_{j=1}^k \tau_j \xi_j\biggr)(\xi_1,\dots, \xi_k, h), \theta, h\in E,
\end{equation}
where $\tau_1,\dots, \tau_k$
are i.i.d. random variables uniformly distributed in $[0,1]$ and independent 
of $\xi_1,\dots, \xi_k.$
\end{theorem}

\begin{proof}
First note that the expression in the right hand side of \eqref{deriv_formula} 
is well defined. This easily follows from the bound 
\begin{align*}
&
\biggl\|f^{(k+1)}\biggl(\theta+ \sum_{j=1}^k \tau_j \xi_j\biggr)(\xi_1,\dots, \xi_k)
\biggr\|
\\
&
\leq \|f^{(k+1)}\|_{L_{\infty,\gamma}}\biggl(1\vee \|\theta\| \vee 
\biggl\|\sum_{j=1}^k\tau_j \xi_j\biggr\|\biggr)^{\gamma}
\|\xi_1\|\dots \|\xi_k\|
\end{align*}
whose right hand side has finite expectation. 
By Lebesgue dominated convergence theorem, this also implies 
the continuity of the function 
$
\theta \mapsto (\B^k f)'(\theta)(h)
$
defined by expression \eqref{deriv_formula}. It remains to show that this expression 
indeed provides the derivative of $\B^k f.$
To this end, observe that 
\begin{align*}
&
f^{(k)}\biggl(\theta+ h+ \sum_{j=1}^k \tau_j \xi_j\biggr)(\xi_1,\dots, \xi_k)-
f^{(k)}\biggl(\theta+ \sum_{j=1}^k \tau_j \xi_j\biggr)(\xi_1,\dots, \xi_k)
\\
&
= \int_0^1 f^{(k+1)}\biggl(\theta+ t h+ \sum_{j=1}^k \tau_j \xi_j\biggr)(\xi_1,\dots, \xi_k, h)
dt,
\end{align*}
which implies 
\begin{align*}
&
f^{(k)}\biggl(\theta+ h+ \sum_{j=1}^k \tau_j \xi_j\biggr)(\xi_1,\dots, \xi_k)-
f^{(k)}\biggl(\theta+ \sum_{j=1}^k \tau_j \xi_j\biggr)(\xi_1,\dots, \xi_k)
\\
&
-f^{(k+1)}\biggl(\theta+\sum_{j=1}^k \tau_j \xi_j\biggr)(\xi_1,\dots, \xi_k, h)
\\
&
= \int_0^1 \biggl[f^{(k+1)}\biggl(\theta+ t h+ \sum_{j=1}^k \tau_j \xi_j\biggr)(\xi_1,\dots, \xi_k, h)
- f^{(k+1)}\biggl(\theta+\sum_{j=1}^k \tau_j \xi_j\biggr)(\xi_1,\dots, \xi_k, h)\biggr]
dt
\end{align*}
and 
\begin{align*}
&
\B^kf(\theta+h)-\B^k f(\theta)- (\B^k f)'(\theta)(h)
\\
&
= \E\int_0^1 \biggl[f^{(k+1)}\biggl(\theta+ t h+ \sum_{j=1}^k \tau_j \xi_j\biggr)(\xi_1,\dots, \xi_k, h)
- f^{(k+1)}\biggl(\theta+\sum_{j=1}^k \tau_j \xi_j\biggr)(\xi_1,\dots, \xi_k, h)\biggr]
dt.
\end{align*}
Therefore, 
\begin{align*}
&
\Bigl|\B^kf(\theta+h)-\B^k f(\theta)- (\B^k f)'(\theta)(h)\Bigr|
\\
&
\leq \E\int_0^1 \biggl \|f^{(k+1)}\biggl(\theta+ t h+ \sum_{j=1}^k \tau_j \xi_j\biggr)
- f^{(k+1)}\biggl(\theta+\sum_{j=1}^k \tau_j \xi_j\biggr)\biggr\| dt\|\xi_1\|\dots \|\xi_k\| \|h\|.
\end{align*}
It remains to observe that by continuity of $f^{(k+1)}$ 
$$
\biggl \|f^{(k+1)}\biggl(\theta+ t h+ \sum_{j=1}^k \tau_j \xi_j\biggr)
- f^{(k+1)}\biggl(\theta+\sum_{j=1}^k \tau_j \xi_j\biggr)\biggr\| \to 0\ {\rm as}\ h\to 0, t\in [0,1],
$$
and to use Lebesgue dominated convergence to conclude that 
$$
\E\int_0^1 \biggl \|f^{(k+1)}\biggl(\theta+ t h+ \sum_{j=1}^k \tau_j \xi_j\biggr)
- f^{(k+1)}\biggl(\theta+\sum_{j=1}^k \tau_j \xi_j\biggr)\biggr\| dt\|\xi_1\|\dots \|\xi_k\|
=o(1)\ {\rm as}\ h\to 0.
$$
This proves Fr\'echet differentiability of the function $\theta \mapsto \B^{k} f(\theta)$
along with formula \eqref{deriv_formula} for its derivative.

\qed
\end{proof}

The following corollary is immediate.

\begin{corollary}
\label{repr_remainder_B^k}
Suppose $f\in C^{k+1,\gamma}(E)$ for some $\gamma\geq 0.$ Then
\begin{equation}
\label{remainder_B^k_formula}
S_{\B^k f}(\theta;h)=
\E S_{f^{(k)}}\biggl(\theta+ \sum_{j=1}^k \tau_j \xi_j; h\biggr)(\xi_1,\dots, \xi_k), 
\theta, h\in E.
\end{equation}
\end{corollary}

\begin{proposition}
\label{B^kf'prop}
Let $s=k+1+\rho$ for some $\rho\in (0,1]$ and let $\gamma\geq 0.$
Suppose that $f\in C^{s,\gamma}(E).$
Then, for all $j=1,\dots, k,$
\begin{align}
\label{B^jf'L}
\|(\B^j f)'\|_{L_{\infty},\gamma} \leq 
2^{\gamma}\|f^{(j+1)}\|_{L_{\infty,\gamma}} 
(1+ j^{\gamma/2}\E^{1/2}\|\xi\|^{2\gamma})
(\E^{1/2}\|\xi\|^2)^j.
\end{align}
Moreover, for all $j=1,\dots, k-1$ 
\begin{align}
\label{B^jf'_norm}
&
\|(\B^j f)'\|_{{\rm Lip}_{1,\gamma}} 
\leq 
2^{\gamma}\|f^{(j+1)}\|_{{\rm Lip}_{1,\gamma}} 
(1+ j^{\gamma/2}\E^{1/2}\|\xi\|^{2\gamma}) 
(\E^{1/2}\|\xi\|^2)^j
\end{align}
and 
\begin{align}
\label{B^kf'_norm}
&
\|(\B^k f)'\|_{{\rm Lip}_{\rho,\gamma}} 
\leq 
2^{\gamma}\|f^{(k+1)}\|_{{\rm Lip}_{\rho,\gamma}} 
(1+ k^{\gamma/2}\E^{1/2}\|\xi\|^{2\gamma}) 
(\E^{1/2}\|\xi\|^2)^k.
\end{align}
\end{proposition}

\begin{proof}
We will prove only the last bound of the proposition. The proof of other bounds 
is similar.  
Using representation \eqref{deriv_formula}, we get
\begin{align*}
&
\|(\B^k f)'(\theta_1)-(\B^k f)'(\theta_2)\| 
\\
&
\leq 
\E \biggl\|f^{(k+1)}\biggl(\theta_1+ \sum_{j=1}^k \tau_j \xi_j\biggr)
-
f^{(k+1)}\biggl(\theta_2+ \sum_{j=1}^k \tau_j \xi_j\biggr)
\biggr\|\|\xi_1\|\dots \|\xi_k\|
\\
&
\leq 2^{\gamma}\|f^{(k+1)}\|_{{\rm Lip}_{\rho,\gamma}} \E\biggl(1\vee \|\theta_1\|\vee \|\theta_2\|\vee 
\biggl\|\sum_{j=1}^k \tau_j \xi_j\biggr\|\biggr)^{\gamma} \|\xi_1\|\dots \|\xi_k\|\|\theta_1-\theta_2\|^{\rho}
\\
&
\leq 2^{\gamma}\|f^{(k+1)}\|_{{\rm Lip}_{\rho,\gamma}} 
\E^{1/2}\biggl(1\vee \|\theta_1\|\vee \|\theta_2\|\vee 
\biggl\|\sum_{j=1}^k \tau_j \xi_j\biggr\|\biggr)^{2\gamma} 
(\E^{1/2}\|\xi\|^2)^k\|\theta_1-\theta_2\|^{\rho}
\end{align*}
Next recall that, conditionally on $\tau_1,\dots, \tau_k,$ $\sum_{j=1}^k \tau_j \xi_j$
has the same distribution as $\Bigl(\sum_{j=1}^k \tau_j^2\Bigr)^{1/2}\xi.$  
Therefore,
\begin{align*}
&
\E\biggl(1\vee \|\theta_1\|\vee \|\theta_2\|\vee 
\biggl\|\sum_{j=1}^k \tau_j \xi_j\biggr\|\biggr)^{2\gamma} 
=
\E\biggl(1\vee \|\theta_1\|\vee \|\theta_2\|\vee 
\Bigl(\sum_{j=1}^k \tau_j^2\Bigr)^{1/2}\|\xi\|
\biggr)^{2\gamma} 
\\
&
\leq 
\E\biggl(1\vee \|\theta_1\|\vee \|\theta_2\|\vee 
k^{1/2}\|\xi\|
\biggr)^{2\gamma} 
\leq 
(1\vee \|\theta_1\|\vee \|\theta_2\|)^{2\gamma} + k^{\gamma}\E\|\xi\|^{2\gamma}.
\end{align*}
Hence, we easily get 
\begin{align*}
&
\|(\B^k f)'(\theta_1)-(\B^k f)'(\theta_2)\| 
\\
&
\leq 
2^{\gamma}\|f^{(k+1)}\|_{{\rm Lip}_{\rho,\gamma}} 
\Bigl[(1\vee \|\theta_1\|\vee \|\theta_2\|)^{\gamma} + k^{\gamma/2}\E^{1/2}\|\xi\|^{2\gamma}\Bigr] 
(\E^{1/2}\|\xi\|^2)^k\|\theta_1-\theta_2\|^{\rho},
\end{align*}
implying the result.
\qed
\end{proof}

\begin{proposition}
\label{bound_B^{k+1}}
Let $s=k+1+\rho$ for some $\rho\in (0,1]$ and let $\gamma\geq 0.$
Suppose that $f\in C^{s,\gamma}(E).$
Then 
\begin{align*}
\|\B^{k+1}f\|_{L_{\infty,\gamma}}
\lesssim 
2^{\gamma}\|f^{(k+1)}\|_{{\rm Lip}_{\rho,\gamma}} 
(1+ k^{\gamma/2}\E^{1/2}\|\xi\|^{2\gamma})
(1+\E^{1/2}\|\xi\|^{2\gamma})
(\E^{1/2}\|\xi\|^2)^{s}. 
\end{align*} 
\end{proposition}

\begin{proof}
Note that 
\begin{align*}
&
\B^{k+1}f(\theta) 
= \B\B^{k}f(\theta)=
\E\B^kf(\theta+\xi)-\B^k f(\theta) 
\\
&
=\E(\B^k f)'(\theta)(\xi)+ \E S_{\B^k}f(\theta;\xi)
=\E S_{\B^k}f(\theta;\xi).
\end{align*}
Using the first bound of Lemma \ref{Taylor_remainder} along with bound 
\eqref{B^kf'_norm}, we get 
\begin{align*}
&
|S_{\B^k f}(\theta;\xi)| 
\lessim 
\|(\B^k f)'\|_{{\rm Lip}_{\rho,\gamma}}
(1\vee \|\theta\|\vee \|\xi\|)^{\gamma} \|\xi\|^{1+\rho} 
\\
&
\lesssim 
2^{\gamma}\|f^{(k+1)}\|_{{\rm Lip}_{\rho,\gamma}} 
(1+ k^{\gamma/2}\E^{1/2}\|\xi\|^{2\gamma}) 
(\E^{1/2}\|\xi\|^2)^k
(1\vee \|\theta\|\vee \|\xi\|)^{\gamma} \|\xi\|^{1+\rho}.
\end{align*}
Therefore,
\begin{align*}
&
|\B^{k+1}f(\theta)| \leq 
\E|S_{\B^k f}(\theta;\xi)|
\\
&
\lessim
2^{\gamma}\|f^{(k+1)}\|_{{\rm Lip}_{\rho,\gamma}} 
(1+ k^{\gamma/2}\E^{1/2}\|\xi\|^{2\gamma}) 
(\E^{1/2}\|\xi\|^2)^k
\E(1\vee \|\theta\|\vee \|\xi\|)^{\gamma} \|\xi\|^{1+\rho}
\\
&
\lessim 
2^{\gamma}\|f^{(k+1)}\|_{{\rm Lip}_{\rho,\gamma}} 
(1+ k^{\gamma/2}\E^{1/2}\|\xi\|^{2\gamma}) 
(\E^{1/2}\|\xi\|^2)^k
\E^{1/2}(1\vee \|\theta\|\vee \|\xi\|)^{2\gamma} \E^{1/2}\|\xi\|^{2(1+\rho)}.
\end{align*}
Since for a centered Gaussian random variable $\xi$ and for $\rho\in (0,1],$ 
$$
\E^{1/2}\|\xi\|^{2(1+\rho)}\lesssim (\E^{1/2}\|\xi\|^2)^{1+\rho},
$$
we get 
\begin{align*}
&
|\B^{k+1}f(\theta)| 
\\
&
\lesssim 
2^{\gamma}\|f^{(k+1)}\|_{{\rm Lip}_{\rho,\gamma}} 
(1+ k^{\gamma/2}\E^{1/2}\|\xi\|^{2\gamma})
\Bigl[(1\vee \|\theta\|)^{\gamma}+\E^{1/2}\|\xi\|^{2\gamma}\bigr]
(\E^{1/2}\|\xi\|^2)^{k+1+\rho}, 
\end{align*}
implying the claim.
\qed
\end{proof}

Theorem \ref{bias_control} immediately follows from the bound of Proposition \ref{bound_B^{k+1}}
and formula \eqref{bias_111}.


\section{Concentration}
\label{sec.conc}

In this section, we prove a concentration inequality for random variable $g(\xi),$
where $\xi $ is a Gaussian random vector in $E$ with mean zero and covariance operator $\Sigma$
and $g$ is a functional on $E$ satisfying the assumption described below. This inequality will be then used 
to prove concentration bounds for estimator $f_k(X).$ 

\begin{assumption}
\label{Lip_lip}
Suppose $g:E\mapsto {\mathbb R}$ satisfies the following Lipschitz condition:
$$
|g(x)-g(x')| \leq L(\|x\|\vee \|x'\|) \|x-x'\|, x,x'\in E,
$$
where $\delta\geq 0\mapsto L(\delta)\in {\mathbb R}_+$ is a non-decreasing continuous function such that 
\begin{align}
\label{assump_L}
L(a\delta) \lesssim_{L} L(a) e^{\delta^2/2}, \delta\geq 0, a\geq 0.
\end{align}
\end{assumption}

It is easy to see that assumption \eqref{assump_L} on function $L$ implies that for any constant $c_1>0$
there exists a constant $c_2>0$ (depending only on $L$) such that $L(c_1 \delta)\leq c_2L(\delta), \delta\geq 0.$
Clearly, \eqref{assump_L} holds for $L(\delta):=C\delta^{\alpha}, \delta\geq 0$
for arbitrary $C>0,  \alpha\geq 0.$ 
Also, if functions $L_1,\dots, L_m$ satisfy  assumption \eqref{assump_L}, then so do the functions 
$L_1+\dots+L_m,$ $L_1\vee \dots \vee L_m.$ 
In particular, this implies that any function of the form 
$$
L(\delta):= C_1 \delta^{\alpha_1}\bigvee \dots \bigvee C_m \delta^{\alpha_m}, \delta\geq 0,
$$
where $m\geq 1,$ $C_1>0,\dots, C_m>0$ and $\alpha_1\geq 0, \dots, \alpha_m\geq 0$
are given constants, satisfy assumption \eqref{assump_L}.

Note that, if $g(0)=0,$ then Assumption \ref{Lip_lip} implies that 
$$
|g(x)|\leq L(\|x\|)\|x\|, x\in E.
$$

\begin{theorem}
Suppose Assumption \ref{Lip_lip} holds. 
For all $t\geq 1$ with probability at least $1-e^{-t},$
\begin{equation}
\label{conc_conc_1}
|g(\xi)-{\mathbb E}g(\xi)|\lesssim  L({\mathbb E} \|\xi\| \vee \|\Sigma\|^{1/2}\sqrt{t})\|\Sigma\|^{1/2}\sqrt{t}. 
\end{equation}
\end{theorem}

\begin{proof}
Without loss of generality, assume that $g(0)=0.$ For $\delta>0,$ define 
\begin{align*}
h(x)&\coloneqq  g(x)\varphi \biggl(\frac{\|x\|}{\delta}\biggr), x\in E,
\end{align*}
where  
\begin{align*}
\varphi(u)&=
\begin{cases}
2-x,& u\in (1,2),\\
1,& u\leq 1,\\
0,& u\geq 2.
\end{cases}
\end{align*}
 Clearly, $\varphi$ is a Lipschitz function with constant $1.$
 We will now prove a Lipschitz condition for the function $h:E\mapsto {\mathbb R}.$ 

\begin{lemma}
Under Assumption \ref{Lip_lip},
for all $x, x^{\prime}\in \EB$
\begin{align*}
|h(x)-h(x^{\prime})|\leq
3 L(3\delta) \|x-x'\|.
\end{align*}
\end{lemma}

\begin{proof}
Note that 
\begin{align}
\label{eq:h_Lip}
|h(x)-h(x^{\prime})|&=|h(x)-h(x^{\prime})| \Ind\{\|x\|\leq 2\delta,\|x^{\prime}\|\leq 2\delta\}
\\&\quad +|h(x)-h(x^{\prime})| \Ind\{\|x\|\leq 2\delta, \|x^{\prime}\| >2 \delta\}
\nonumber
\\&\quad +|h(x)-h(x^{\prime})| \Ind\{\|x^{\prime}\|\leq 2\delta, \|x\|>2 \delta\}.
\nonumber
\end{align}
For the first summand in the right hand side of \eqref{eq:h_Lip}, we have the following bound for $\|x\|,\|x^{\prime}\|\leq 2\delta$
\begin{align}
\nonumber
&|h(x)-h(x^{\prime})| \Ind\{\|x\|,\|x^{\prime}\|\leq 2\delta\}
\\
\nonumber
&\leq
|g(x)-g(x^{\prime})|\varphi \biggl(\frac{\|x\|}{\delta}\biggr)+
|g(x^{\prime})|\left|\varphi \biggl(\frac{\|x\|}{\delta}\biggr)-\varphi \biggl(\frac{\|x^{\prime}\|}{\delta}\biggr)\right|
\\
\nonumber
&\leq L(2\delta)\|x-x^{\prime}\|+L(2\delta)\|x^{\prime}\|\|x-x^{\prime}\|/\delta
\\&\leq 3  L(2\delta)\|\|x-x^{\prime}\|.
\label{ineq:h_1termbound}
\end{align}
To bound the second summand in \eqref{eq:h_Lip}, observe that for $\|x\|\leq 2\delta, \|x^{\prime}\|> 2\delta$
\begin{align}
\nonumber
&|h(x)-h(x^{\prime})|\Ind\{\|x\|\leq 2\delta, \|x^{\prime}\| >2 \delta\}
\\
\label{eq:h_2term}
&=
|h(x)-h(x^{\prime})| \Ind\{\|x\|\leq 2\delta, \|x^{\prime}\| >2 \delta, \|x-x^{\prime}\| \geq\delta \} 
\\&\quad+ |h(x)-h(x^{\prime})| \Ind\{\|x\|\leq 2\delta, \|x^{\prime}\| >2 \delta,  \|x-x^{\prime}\| <\delta\},
\nonumber
\end{align}
and bound the first term in the right hand side of \eqref{eq:h_2term} as follows:
\begin{align*}
&
|h(x)-h(x^{\prime})|
\Ind\{\|x\|\leq 2\delta, \|x^{\prime}\| >2 \delta, \|x-x^{\prime}\| \geq \delta \}
\\&=
|g(x)|\varphi \biggl(\frac{\|x\|}{\delta}\biggr) \Ind\{\|x\|\leq 2\delta, \|x^{\prime}\| >2 \delta, \|x-x^{\prime}\| \geq\delta \}
\\&\leq 
L(2\delta)\|x\| \Ind\{\|x\|\leq 2\delta, \|x^{\prime}\| >2 \delta, \|x-x^{\prime}\| \geq\delta \}.
\\&\leq  2L(2\delta)\|x-x^{\prime}\|.
\end{align*}
For the second term in \eqref{eq:h_2term}, we have
\begin{align*}
&
|h(x)-h(x^{\prime})|
\Ind\{\|x\|\leq 2\delta, \|x^{\prime}\| >2 \delta,  \|x-x^{\prime}\| <\delta\}
\\&= 
|h(x)-h(x^{\prime})| \Ind\{\|x\|\leq 2\delta, 2\delta<\|x^{\prime}\|\leq 3 \delta,  \|x-x^{\prime}\| <\delta\}
\\&\leq|h(x)-h(x^{\prime})| \Ind\{\|x\|\leq 3\delta,\|x^{\prime}\|\leq 3\delta\}
\\&\leq 3  L(3\delta)\|\|x-x^{\prime}\|,
\end{align*}
where the last inequality is proved similarly to bound \eqref{ineq:h_1termbound} 
(with an obvious change of $2\delta$ to $3\delta$). Substituting the above bounds in \eqref{eq:h_Lip}, leads to the resulting inequality.
\QED
\end{proof}

In what follows, we set 
$$
\delta = \delta (t):={\mathbb E}\|\xi\| + C\|\Sigma\|^{1/2}\sqrt{t}
$$
for $t\geq 1$ with a constant $C>0$ such that 
$$
{\mathbb P}\{\|\xi\|\geq \delta(t)\}\leq e^{-t}, t\geq 1
$$
(which holds by the Gaussian concentration inequality, see, e.g.,  \cite{Ledoux}).

Let $M:={\rm Med}(g(\xi)).$ Assuming that $t\geq \log (4),$ we get
\begin{align*}
&
{\mathbb P}\{h(\xi)\geq M\}\geq {\mathbb P}\{h(\xi)\geq M, \|\xi\|\leq \delta(t)\}
\geq {\mathbb P}\{g(\xi)\geq M, \|\xi\|\leq \delta(t)\}
\\
&
\geq {\mathbb P}\{g(\xi)\geq M\}-
{\mathbb P}\{\|\xi\|\geq \delta(t)\}\geq \frac{1}{2}-e^{-t}\geq \frac{1}{4},
\end{align*}
where we used the fact that, on the event $\{\|\xi\|\leq \delta\},$
$h(\xi)=g(\xi).$ Similarly, we have 
$
{\mathbb P}\{h(\xi)\leq M\}\geq \frac{1}{4}.
$
We can now use again Gaussian concentration inequality (in a little bit non-standard fashion, see \cite{Koltchinskii_Lounici_arxiv}, Section 3 for a similar argument) to prove that with probability 
at least $1-e^{-t}$
$$
|h(\xi)-M| \lesssim L(3\delta(t)) \|\Sigma\|^{1/2}\sqrt{t}
$$
and, since $h(\xi)$ and $g(\xi)$ coincide on the event of probability at least 
$1-e^{-t},$ we also have that 
$$
|g(\xi)-M| \lesssim L(3\delta(t)) \|\Sigma\|^{1/2}\sqrt{t}
$$
with probability at least $1-2e^{-t}.$ Moreover, by adjusting 
the value of the constant in the above inequality, the probability 
bound can be written in its standard form $1-e^{-t}$ and the inequality 
holds for all $t\geq 1.$ Using the properties of function $L$ (namely, its 
monotonicity and condition \eqref{assump_L})
and the definition of $\delta(t),$
we can also rewrite the above bound as 
$$
|g(\xi)-M| \leq C_L (L({\mathbb E}\|\xi\|) \|\Sigma\|^{1/2}\sqrt{t}
\bigvee L(\|\Sigma\|^{1/2}\sqrt{t}) \|\Sigma\|^{1/2}\sqrt{t})=: s(t)
$$
for some constant $C_L>0.$
Note that this bound actually holds for all $t\geq 0$ with probability 
at least $1-e^{1-t}.$ Note also that the function $t\mapsto s(t)$
is strictly increasing on $[0,+\infty)$ with $s(0)=0$ and $s(+\infty)=+\infty.$
Moreover, it easily follows from condition \eqref{assump_L} that 
$s(t)=o(e^{t})$ as $t\to \infty.$
It remains to integrate out the tails of the probability bound: 
\begin{align*}
&
|{\mathbb E} g(\xi)-M|\leq {\mathbb E}|g(\xi)-M|
\\
&
= \int_0^{\infty}{\mathbb P}\{|g(\xi)-M|\geq s\}ds
=\int_0^{\infty}{\mathbb P}\{|g(\xi)-M|\geq s(t)\}ds(t)
\\
&
\leq e\int_{0}^{\infty} e^{-t}ds(t)= 
e\int_{0}^{\infty}s(t)e^{-t}dt.
\end{align*}
By condition \eqref{assump_L}, 
$$
s(t) \leq C_L L({\mathbb E}\|\xi\|) \|\Sigma\|^{1/2}\sqrt{t}+ C_L'  L(\|\Sigma\|^{1/2}) 
\|\Sigma\|^{1/2}\sqrt{t} e^{t/2}, t\geq 0.
$$
Therefore,
\begin{align*}
&
|{\mathbb E} g(\xi)-M|\leq \int_{0}^{\infty}s(t)e^{-t}dt
\\
&
\lesssim_L 
L({\mathbb E}\|\xi\|) \|\Sigma\|^{1/2}\int_{0}^{\infty}\sqrt{t}e^{-t}dt
+L(\|\Sigma\|^{1/2}) \|\Sigma\|^{1/2}\int_{0}^{\infty}\sqrt{t}e^{-t/2}dt
\\
&
\lesssim_L 
L({\mathbb E}\|\xi\|) \|\Sigma\|^{1/2}\bigvee L(\|\Sigma\|^{1/2}) \|\Sigma\|^{1/2},
\end{align*}
which now allows us to replace the median $M$ by the mean ${\mathbb E} g(\xi)$
in the concentration bound, completing the proof.
\qed
\end{proof}

The following corollary is immediate (for the proof, check that Assumption \ref{Lip_lip}
holds with $L(\delta)=C\|g\|_{{\rm Lip}_{1,\gamma}}(1\vee \|\theta\|\vee \delta)^{\gamma}$
for some $C>0$). 

\begin{corollary}
\label{conc_g}
Suppose $g\in {\rm Lip}_{1,\gamma}(E)$ for some $\gamma\geq 0.$ Then, for all $\theta\in E$
and for all $t\geq 1$ with probability at least $1-e^{-t}$
$$
|g(\theta+\xi)-{\mathbb E}g(\theta+\xi)|\lesssim \|g\|_{{\rm Lip}_{1,\gamma}}
(1\vee \|\theta\| \vee {\mathbb E}\|\xi\|\vee \|\Sigma\|^{1/2}\sqrt{t})^{\gamma}\|\Sigma\|^{1/2}
\sqrt{t}.
$$
\end{corollary}

Another immediate corollary of this theorem is the following concentration bound for the remainder $S_g(\theta;\xi)$ of the first order Taylor expansion of 
$g(\theta+\xi).$ For the proof, it is enough to observe that, by Lemma \ref{Taylor_remainder},
the function $x\mapsto S_g(\theta;x)$ satisfies Assumption \ref{Lip_lip} with 
$L(\delta)= C\|g'\|_{{\rm Lip}_{\rho,\gamma}} (1\vee \|\theta\|\vee \delta)^{\gamma} \delta^{\rho}$
for some constant $C>0.$

\begin{corollary}
\label{conc_S_g}
Suppose, for some $\rho\in (0,1]$ and $\gamma\geq 0,$ $\|g'\|_{{\rm Lip}_{\rho,\gamma}}<\infty.$
Then, for all $\theta\in E$ and for all $t\geq 1$ with probability at least $1-e^{-t},$
$$
|S_g(\theta;\xi)-{\mathbb E}S_g(\theta;\xi)|\lesssim \|g'\|_{{\rm Lip}_{\rho,\gamma}} (1\vee \|\theta\|\vee {\mathbb E}\|\xi\|\vee \|\Sigma\|^{1/2}\sqrt{t})^{\gamma}({\mathbb E}\|\xi\| \vee \|\Sigma\|^{1/2}\sqrt{t})^{\rho}\|\Sigma\|^{1/2}\sqrt{t}.
$$ 
\end{corollary}

We now apply corollaries \ref{conc_g} and \ref{conc_S_g} to obtain concentration bounds 
for estimator $f_k(X)$ and the remainder of its first order Taylor expansion.

\begin{proposition}
Let $\gamma\geq 0$ and suppose that $f\in C^{k+1,\gamma}(E)$ and that ${\mathbb E}^{1/2}\|\xi\|^2\leq 1/2.$
Then, for all $t\geq 1,$ with probability 
at least $1-e^{-t}$
\begin{align}
\label{conc_f_k}
|f_k(\theta+\xi)-{\mathbb E}f_k(\theta+\xi)|\lesssim_{\gamma} (k+1)^{\gamma/2}\|f\|_{C^{k+1,\gamma}}
(1\vee \|\theta\| \vee \|\Sigma\|^{1/2}\sqrt{t})^{\gamma}\|\Sigma\|^{1/2}
\sqrt{t}.
\end{align}
\end{proposition}

\begin{proof}
Using bound \eqref{B^jf'L}, we get 
\begin{align*}
&
\|f_k'\|_{L_{\infty},\gamma}\leq 
\sum_{j=0}^k \|({\mathcal B}^j f)'\|_{L_{\infty},\gamma}
\\
&
\leq 2^{\gamma}\sum_{j=0}^k \|f^{(j+1)}\|_{L_{\infty,\gamma}} 
(1+ j^{\gamma/2}{\mathbb E}^{1/2}\|\xi\|^{2\gamma})
({\mathbb E}^{1/2}\|\xi\|^2)^j
\\
&
\leq 2^{\gamma} \|f\|_{C^{k+1,\gamma}}(1+ k^{\gamma/2}{\mathbb E}^{1/2}\|\xi\|^{2\gamma})
\sum_{j=0}^k 
({\mathbb E}^{1/2}\|\xi\|^2)^j
\\
&
\leq 2^{\gamma+1} \|f\|_{C^{k+1,\gamma}}(1+ k^{\gamma/2}{\mathbb E}^{1/2}\|\xi\|^{2\gamma})
\lesssim_{\gamma} (k+1)^{\gamma/2}\|f\|_{C^{k+1,\gamma}}.
\end{align*}
The result now follows from Corollary \ref{conc_g}. 
\qed
\end{proof}

With a little additional work, we get the following modification of concentration 
bound \eqref{conc_f_k}.

\begin{corollary}
Let $\gamma\geq 0$ and suppose that $f\in C^{k+1,\gamma}(E)$ and that ${\mathbb E}^{1/2}\|\xi\|^2\leq 1/2.$
If $\gamma\leq 1,$ then, for all $t\geq 1,$ with probability 
at least $1-e^{-t}$
\begin{align}
\label{conc_f_k_A}
|f_k(\theta+\xi)-{\mathbb E}f_k(\theta+\xi)|\lesssim_{\gamma} 
(k+1)^{\gamma/2}\|f\|_{C^{k+1,\gamma}}
(1\vee \|\theta\|)^{\gamma}\|\Sigma\|^{1/2}\sqrt{t}.
\end{align}
If $\gamma>1,$ then, for all $t\geq 1,$ with the same probability
\begin{align}
\label{conc_f_k_B}
|f_k(\theta+\xi)-{\mathbb E}f_k(\theta+\xi)|\lesssim_{\gamma} 
(k+1)^{\gamma/2}\|f\|_{C^{k+1,\gamma}}
(1\vee \|\theta\|)^{\gamma}\Bigl(\|\Sigma\|^{1/2}\sqrt{t}\vee (\|\Sigma\|^{1/2}\sqrt{t})^{\gamma}\Bigr).
\end{align}
\end{corollary}

\begin{proof}
It follows from Corollary \ref{cor_f_k_infty_gamma} that 
$
\|f_k\|_{L_{\infty,\gamma}} 
\lesssim_{\gamma} (k+1)^{\gamma/2} \|f\|_{C^{k, \gamma}}.
$
This implies that 
$$
|f_k(\theta+\xi)|\lesssim_{\gamma} 
(k+1)^{\gamma/2} \|f\|_{C^{k, \gamma}}(1\vee \|\theta\|\vee \|\xi\|)^{\gamma},
$$
which easily yields the following bounds
$$
|{\mathbb E}f_k(\theta+\xi)|\lesssim_{\gamma} 
(k+1)^{\gamma/2} \|f\|_{C^{k, \gamma}}(1\vee \|\theta\|)^{\gamma}
$$
and 
$$
|f_k(\theta+\xi)|\lesssim_{\gamma} 
(k+1)^{\gamma/2} \|f\|_{C^{k, \gamma}}(1\vee \|\theta\| \vee \|\Sigma\|^{1/2}\sqrt{t})^{\gamma}
$$
(the last bound holds for all $t\geq 1$ with probability at least $1-e^{-t}$).
Therefore, for all $t\geq 1$ with probability at least $1-e^{-t}$
\begin{equation}
\label{conc_f_k_C}
|f_k(\theta+\xi)-{\mathbb E}f_k(\theta+\xi)|\lesssim_{\gamma} 
(k+1)^{\gamma/2} \|f\|_{C^{k, \gamma}}(1\vee \|\theta\| \vee \|\Sigma\|^{1/2}\sqrt{t})^{\gamma}.
\end{equation}
If $\|\Sigma\|^{1/2}\sqrt{t}\leq 1,$ then bound \eqref{conc_f_k_A} follows from bound \eqref{conc_f_k} (regardless of what the value of $\gamma\geq 0$ is). If $\|\Sigma\|^{1/2}\sqrt{t}>1,$ we use bound \eqref{conc_f_k_C} to get 
\begin{align*}
&
|f_k(\theta+\xi)-{\mathbb E}f_k(\theta+\xi)|
\\
&
\lesssim_{\gamma} 
(k+1)^{\gamma/2} \|f\|_{C^{k, \gamma}}(1\vee \|\theta\|)^{\gamma}\bigvee (k+1)^{\gamma/2} \|f\|_{C^{k, \gamma}}(\|\Sigma\|^{1/2}\sqrt{t})^{\gamma}
\\
&
\lesssim_{\gamma}
(k+1)^{\gamma/2} \|f\|_{C^{k, \gamma}}(1\vee \|\theta\|)^{\gamma}(\|\Sigma\|^{1/2}\sqrt{t})^{\gamma},
\end{align*}
which yields \eqref{conc_f_k_B}.
\qed
\end{proof}

Given an increasing, convex function $\psi : {\mathbb R}_+\mapsto {\mathbb R}_+$ with $\psi(0)=0$
and $\psi(t)\to +\infty$ as $t\to+\infty$ (in what follows, an Orlicz function), 
the Orlicz $\psi$-norm of a r.v. $\eta$ is defined as 
$$
\|\eta\|_{\psi}:=\inf \biggl\{C>0: {\mathbb E}\psi\biggl(\frac{|\eta|}{C}\biggr)\leq 1\biggr\}.
$$ 
For $p\geq 1$ and $\psi(t):=t^p, t\geq 0,$ this yields the usual $L_p$-norms.
Another popular choice is $\psi_{\alpha}(t):=e^{t^{\alpha}}-1, t\geq 0$ for some 
$\alpha\geq 1,$ in particular, $\psi_2$-norm for subgaussian random variables 
and $\psi_1$-norm for subexponential random variables. 

We will need the following simple lemma.

\begin{lemma}
\label{Orlicz}
Let $Y$ be a non-negative random variable. Suppose, for some $A_1>0,\dots, A_m>0, \beta_1>0, \dots, \beta_m>0$
and for all $t\geq 1,$
$$
{\mathbb P}\{Y\geq A_1 t^{\beta_1}\vee \dots \vee A_m t^{\beta_m}\}\leq e^{-t}.
$$
Let $\beta:=\max_{1\leq j\leq m}\beta_j.$ Then, for any Orlicz function $\psi$ satisfying 
the condition $\psi(t)\leq c_1 e^{c_2 t^{1/\beta}}, t\geq 0$ for some constants $c_1,c_2>0,$ we have
$$
\|Y\|_{\psi}\lesssim_{\psi} A_1\vee \dots \vee A_m.
$$
\end{lemma}

\begin{proposition}
Let $s=k+1+\rho$ for some $\rho\in (0,1]$ and let $\gamma\geq 0.$
Suppose that $f\in C^{s,\gamma}(E)$ and that ${\mathbb E}^{1/2}\|\xi\|^2\leq 1/2.$
If $\gamma\leq 1,$ then, for all $t\geq 1,$ with probability 
at least $1-e^{-t}$
\begin{align}
\label{conc_f_k_A_A}
|f_k(\theta+\xi)-f(\theta)|\lesssim_{\gamma} 
(k+1)^{\gamma/2}\|f\|_{C^{s,\gamma}}
(1\vee \|\theta\|)^{\gamma}\Bigl(\|\Sigma\|^{1/2}\sqrt{t}\vee ({\mathbb E}^{1/2}\|\xi\|^2)^s\Bigr).
\end{align}
If $\gamma>1,$ then, for all $t\geq 1,$ with the same probability
\begin{align}
\label{conc_f_k_B_B}
|f_k(\theta+\xi)-f(\theta)|\lesssim_{\gamma} 
(k+1)^{\gamma/2}\|f\|_{C^{s,\gamma}}
(1\vee \|\theta\|)^{\gamma}\Bigl(\|\Sigma\|^{1/2}\sqrt{t}\vee (\|\Sigma\|^{1/2}\sqrt{t})^{\gamma}\vee ({\mathbb E}^{1/2}\|\xi\|^2)^s\Bigr).
\end{align}
Moreover, for all Orlicz functions $\psi$ satisfying the condition 
$\psi(t)\leq c_1 e^{c_2 t^{2/(\gamma \vee 1)}}$ for all $t \geq 0$ and for some 
constants $c_1,c_2>0,$
the following bound holds:
\begin{align}
\label{orlicz_AAA}
\|f_k(\theta+\xi)-f(\theta)\|_{\psi}\lesssim_{\gamma,\psi} (k+1)^{\gamma/2} \|f\|_{C^{s, \gamma}}(1\vee \|\theta\|)^{\gamma}\Bigl(\|\Sigma\|^{1/2}\vee ({\mathbb E}^{1/2}\|\xi\|^2)^s\Bigr).
\end{align}
\end{proposition}

\begin{proof}
The proof immediately follows from bounds \eqref{conc_f_k_A}, \eqref{conc_f_k_B}, Lemma \ref{Orlicz}  
and bound on the bias of Theorem \ref{bias_control}.

\qed
\end{proof}

Now it is easy to prove Theorem \ref{first_result} stated in Section \ref{Overview}.

\begin{proof}
If ${\mathbb E}^{1/2}\|\xi\|^2\leq 1/2,$ the result follows from bound \eqref{orlicz_AAA} (with $\psi(t)=t^2$).
Otherwise, it follows from the bound 
$
|f(\theta)| \leq \|f\|_{C^{s,\gamma}}(1\vee \|\theta\|)^{\gamma}.
$

\qed
\end{proof}

\begin{remark}
\normalfont
In the case when the functional $f:E\mapsto {\mathbb R}$ is a bounded polynomial of degree $k+1,$
estimator $f_k(X)$ is unbiased (see Corollary \ref{bias_polynomial}) and the following version 
of bound \eqref{orlicz_AAA} for $\psi(t)=t^2$ holds (the proof follows the same lines as the proof of
\eqref{orlicz_AAA} with minor modifications):
\begin{align}
\label{bd_poly}
{\mathbb E}_{\theta}(f_k(X)-f(\theta))^2\lesssim_{k} \|f\|_{\rm op}^2
(1\vee \|\theta\|)^{2k}\|\Sigma\|\Bigl(1\vee ({\mathbb E}\|\xi\|^2)^k\Bigr).
\end{align}
In the case of standard Gaussian shift model (see Example \ref{Example 1} in Section \ref{Intro}) and $f(\theta)=\|\theta\|^2$
(a polynomial of degree $2$),
it is easy to check that $f_1(X)= \|X\|^2 -\sigma^2 d.$ Then, bound \eqref{bd_poly} yields that
$$
\sup_{\|\theta\|\leq 1}{\mathbb E}_{\theta}(f_1(X)-f(\theta))^2\lesssim 
\sigma^2\Bigl(1\vee \sigma^2 d\Bigr),
$$
which could be also proved by elementary analysis. 
\end{remark}

The following proposition provides a concentration bound on the remainder $S_{f_{k}}(\theta;\xi)$
of Taylor expansion of function $f_k(\theta +\xi)$ (at point $\theta$). It will be used in the proof of the 
efficiency of estimators $f_k(X).$ 

\begin{proposition}
Let $s=k+1+\rho$ for some $\rho\in (0,1]$ and let $\gamma\geq 0.$
Suppose that $f\in C^{s,\gamma}(E)$ and that ${\mathbb E}^{1/2}\|\xi\|^2\leq 1/2.$
Then, for all $t\geq 1,$ with probability 
at least $1-e^{-t}$
\begin{align}
\label{conc_bound_f_k}
&
\nonumber
|S_{f_{k}}(\theta;\xi)-{\mathbb E}S_{f_{k}}(\theta;\xi)|
\\
&
\lesssim_{\gamma} 
(k+1)^{\gamma/2} \|f\|_{C^{s,\gamma}} 
(1\vee \|\theta\|\vee \|\Sigma\|^{1/2}\sqrt{t})^{\gamma}
(({\mathbb E}\|\xi\|)^{\rho} \vee (\|\Sigma\|^{1/2}\sqrt{t})^{\rho} \vee \|\Sigma\|^{1/2}\sqrt{t})\|\Sigma\|^{1/2}\sqrt{t}.
\end{align}
\end{proposition}

\begin{proof}
It follows from bounds \eqref{B^jf'_norm} of 
Proposition \ref{B^kf'prop} that
\begin{align*}
&
\|f_{k-1}'\|_{{\rm Lip}_{1,\gamma}}
\leq \|f'\|_{{\rm Lip}_{1,\gamma}}+ \sum_{j=1}^{k-1}\|({\mathcal B}^j f)'\|_{{\rm Lip}_{1,\gamma}} 
\\
&
 \leq \|f'\|_{{\rm Lip}_{1,\gamma}}+
2^{\gamma}\sum_{j=1}^{k-1}\|f^{(j+1)}\|_{{\rm Lip}_{1,\gamma}} 
(1+ j^{\gamma/2}{\mathbb E}^{1/2}\|\xi\|^{2\gamma}) 
({\mathbb E}^{1/2}\|\xi\|^2)^j
\\
&
\leq 
\|f''\|_{L_{\infty,\gamma}}+
2^{\gamma}\sum_{j=1}^{k-1}\|f^{(j+2)}\|_{L_{\infty,\gamma}} 
(1+ j^{\gamma/2}{\mathbb E}^{1/2}\|\xi\|^{2\gamma}) 
({\mathbb E}^{1/2}\|\xi\|^2)^j
\\
&
\leq 
\|f\|_{C^{k+1,\gamma}} \biggl(1+2^{\gamma}(1+ k^{\gamma/2}{\mathbb E}^{1/2}\|\xi\|^{2\gamma})
\sum_{j=1}^{k-1} ({\mathbb E}^{1/2}\|\xi\|^2)^j\biggr)
\\
&
\leq 
2^{\gamma+2}(1+ k^{\gamma/2}{\mathbb E}^{1/2}\|\xi\|^{2\gamma})\|f\|_{C^{k+1,\gamma}}.
\end{align*}
Using the  bound of Corollary \ref{conc_S_g}, we get that for all $t\geq 1$ with probability at least $1-e^{-t},$
\begin{align*}
&
|S_{f_{k-1}}(\theta;\xi)-{\mathbb E}S_{f_{k-1}}(\theta;\xi)|
\\
&
\lesssim 
2^{\gamma+2}\|f\|_{C^{k+1,\gamma}} (1+ k^{\gamma/2}{\mathbb E}^{1/2}\|\xi\|^{2\gamma})
(1\vee \|\theta\|\vee {\mathbb E}\|\xi\|\vee \|\Sigma\|^{1/2}\sqrt{t})^{\gamma}({\mathbb E}\|\xi\| \vee \|\Sigma\|^{1/2}\sqrt{t})\|\Sigma\|^{1/2}\sqrt{t}
\\
&
\lesssim_{\gamma} 
(k+1)^{\gamma/2} \|f\|_{C^{k+1,\gamma}} 
(1\vee \|\theta\|\vee \|\Sigma\|^{1/2}\sqrt{t})^{\gamma}({\mathbb E}\|\xi\| \vee \|\Sigma\|^{1/2}\sqrt{t})\|\Sigma\|^{1/2}\sqrt{t}.
\end{align*}
Similarly, using the bound of Corollary \ref{conc_S_g} along with bound \eqref{B^kf'_norm} of Proposition \ref{B^kf'prop}, we get that with probability 
at least $1-e^{-t}$
\begin{align*}
&
|S_{{\mathcal B}^k f}(\theta;\xi)-{\mathbb E}S_{{\mathcal B}^k f}(\theta;\xi)|
\\
&
\lesssim_{\gamma} 
(k+1)^{\gamma/2}\|f^{(k+1)}\|_{{\rm Lip}_{\rho,\gamma}} 
({\mathbb E}^{1/2}\|\xi\|^2)^k
(1\vee \|\theta\|\vee \|\Sigma\|^{1/2}\sqrt{t})^{\gamma}
({\mathbb E}\|\xi\| \vee \|\Sigma\|^{1/2}\sqrt{t})^{\rho}\|\Sigma\|^{1/2}\sqrt{t}.
\end{align*} 
Combining these bounds and adjusting the constants yield bound \eqref{conc_bound_f_k}.
\qed
\end{proof}


\section{Normal Approximation Bounds}
\label{sec:norm_appr}

In this section, we develop normal approximation bounds for $f_k(X)-f(\theta)$ needed to complete the proof of Theorem \ref{norm_appr_th_1}. 
More precisely, it will be shown that $f_k(X)-f(\theta)$ could be approximated by a mean zero normal random 
variable with variance 
$
\sigma_{f,\xi}^2 (\theta):= {\mathbb E}(f'(\theta)(\xi))^2 
=\langle \Sigma f'(\theta), f'(\theta)\rangle.
$
Recall that 
$$
K(f;\Sigma;\theta):=K_{s,\gamma}(f;\Sigma;\theta):=
\frac{\|f\|_{C^{s,\gamma}}(1\vee \|\theta\|)^{\gamma}\|\Sigma\|^{1/2}}{\sigma_{f,\xi}(\theta)}.
$$

\begin{theorem}
\label{norm_appr_CCCC}
Suppose, for some $s=k+1+\rho,$ $\rho\in (0,1]$ and some $\gamma\geq 0,$ $f\in C^{s,\gamma}(E).$
Suppose also that  ${\mathbb E}^{1/2}\|\xi\|^2\leq 1/2.$ Then, the following representation holds  
\begin{align}
\label{Z_R}
f_k(X) - f(\theta)
=\sigma_{f,\xi}(\theta) Z+ R,
\end{align}
where $Z$ is a standard normal random variable and $R$ is 
the remainder satisfying, for all $t\geq 1$ with probability at least $1-e^{-t},$ the bound
\begin{align}
\label{bd_R}
&
\nonumber
|R|
\lesssim_{\gamma} 
(k+1)^{\gamma/2} \|f\|_{C^{s,\gamma}} 
(1\vee \|\theta\|\vee \|\Sigma\|^{1/2}\sqrt{t})^{\gamma}
\\
&
\biggl(
({\mathbb E}^{1/2}\|\xi\|^2)^{\rho} \|\Sigma\|^{1/2}\sqrt{t}
\vee (\|\Sigma\|^{1/2}\sqrt{t})^{1+\rho} \vee \|\Sigma\|t\vee ({\mathbb E}^{1/2}\|\xi\|^2)^{s}
\biggr).
\end{align}
Moreover, for any Orlicz function $\psi$ such that 
$\psi(t)\lesssim  c_1 e^{c_2 t^{2/(2+\gamma)}}, t\geq 0$ 
for some constants $c_1, c_2>0,$
\begin{align}
\label{psi_remainder}
&
\nonumber
\biggl\|\frac{f_k(X) - f(\theta)}{\sigma_{f,\xi}(\theta)}-Z\biggr\|_{\psi}
\\
&
\lesssim_{\gamma,\psi}
(k+1)^{\gamma/2} K_{s,\gamma}(f;\Sigma;\theta)
\biggl(({\mathbb E}^{1/2}\|\xi\|^2)^{\rho}
\bigvee \frac{({\mathbb E}^{1/2}\|\xi\|^2)^s}{\|\Sigma\|^{1/2}}
\biggr).
\end{align}
\end{theorem}

\begin{remark}
\normalfont
Note that ${\mathbb E}\|\xi\|^2 = \|\Sigma\| {\bf r}(\Sigma),$ where ${\bf r}(\Sigma)$ 
is the effective rank of $\Sigma.$ Assume that $\|\Sigma\|$ is ``small" (that is, the noise level
is small) and, for some $\alpha\in (0,1),$
$
{\bf r}(\Sigma)\lesssim \|\Sigma\|^{-\alpha}.
$ 
Then 
$
{\mathbb E}\|\xi\|^2 \lesssim \|\Sigma\|^{1-\alpha},
$
which is ``small", too. Moreover, under the assumption that $s>\frac{1}{1-\alpha},$ 
$$
\frac{({\mathbb E}^{1/2}\|\xi\|^2)^{s}}{\|\Sigma\|^{1/2}}
\lesssim \sqrt{\frac{\|\Sigma\|^{s(1-\alpha)}}{\|\Sigma\|}}
=\sqrt{\|\Sigma\|^{s(1-\alpha)-1}}
$$
is also ``small", implying that the right hand side of bound \eqref{psi_remainder}
is ``small". The same conclusion holds for the right hand side of bound \eqref{norm_appr} 
provided that $K_{s,\gamma}(f;\Sigma;\theta)\lesssim 1.$ 
\end{remark}

\begin{remark}
\normalfont
We will also state (without providing a proof) the following bound on the risk of estimator $f_k(X)$ with 
respect to convex loss functions (under some constraints on their growth rate). 
Let $\ell : {\mathbb R}\mapsto {\mathbb R}_+$ be a loss function such that $\ell (-t)=\ell (t), t\in {\mathbb R},$ 
$\ell $ is an Orlicz function on ${\mathbb R}_+$ and, for some $\delta\in (0,1),$ 
$\nu<\frac{2}{\gamma \vee 1}$
\begin{align}
\label{assump_on_loss_ell}
\ell(t)\lesssim e^{(1-\delta)t^{\nu}}, t \geq 0.
\end{align}
Suppose also that 
\begin{align}
\label{cond_s_1}
\Bigl({\mathbb E}^{1/2}\|\xi\|^2\Bigr)^s\leq \|\Sigma\|^{1/2}. 
\end{align}
Then 
\begin{align}
\label{norm_loss_approx}
&
\nonumber
\biggl|{\mathbb E}\ell\biggl(\frac{f_k(X) - f(\theta)}{\sigma_{f,\xi}(\theta)}\biggr)-{\mathbb E}\ell(Z)\biggr|
\\
&
\lesssim_{\gamma,\ell,\delta}
(k+1)^{\gamma/2} K_{s,\gamma,\ell,k}(f;\Sigma;\theta)
\biggl(({\mathbb E}^{1/2}\|\xi\|^2)^{\rho}
\bigvee \frac{({\mathbb E}^{1/2}\|\xi\|^2)^s}{\|\Sigma\|^{1/2}}
\biggr),
\end{align}
where 
$$
K_{s,\gamma,\ell,k}(f;\Sigma;\theta):= 
(k+1)^{\gamma/2} K_{s,\gamma}(f;\Sigma;\theta)
\biggl(\ell \biggl(c_{\gamma,\nu} (k+1)^{\frac{\gamma}{2-(\gamma\vee 1) \nu}}K_{s,\gamma}^{\frac{1}{1-(\gamma\vee 1) \nu/2}}(f;\Sigma;\theta)\biggr)+1\biggr).
$$
\end{remark}

Bound \eqref{L_2_remainder} of Theorem \ref{norm_appr_th_1} follows from bound 
\eqref{psi_remainder} of Theorem \ref{norm_appr_CCCC} (for $\psi(t)=t^2$). 
We now turn to the proof of Theorem \ref{norm_appr_CCCC} and bound \eqref{norm_appr} of Theorem \ref{norm_appr_th_1}.

\begin{proof}
Clearly, 
\begin{align*}
&
f_k(X) - f(\theta)
\\
&
= f_k(X)-{\mathbb E}_{\theta}f_k(X)
+ {\mathbb E}_{\theta} f_k(X)-f(\theta)
\\
&
= f_k'(\theta)(\xi) + S_{f_k}(\theta;\xi)-{\mathbb E}S_{f_k}(\theta;\xi)
+ {\mathbb E}_{\theta} f_k(X)-f(\theta)
\\
&
=\sigma_{f_k, \xi}(\theta)Z + S_{f_k}(\theta;\xi)-{\mathbb E}S_{f_k}(\theta;\xi)
+ {\mathbb E}_{\theta} f_k(X)-f(\theta)
\\
&
=\sigma_{f,\xi}(\theta) Z+ R,
\end{align*}
where $Z$ is a standard normal random variable and 
\begin{align}
\label{represent_R}
&
R:=
(\sigma_{f_k, \xi}(\theta)-\sigma_{f,\xi}(\theta))Z + S_{f_k}(\theta;\xi)-{\mathbb E}S_{f_k}(\theta;\xi)
+ {\mathbb E}_{\theta} f_k(X)-f(\theta)
\end{align}
is the remainder.

The following lemma will be used to control $\sigma_{f_k,\xi}(\theta)-\sigma_{f,\xi}(\theta).$

\begin{lemma}
Suppose that, for some $\gamma\geq 0,$ $f\in C^{k+1,\gamma}(E)$ and ${\mathbb E}^{1/2}\|\xi\|^2\leq 1/2.$
Then 
\begin{align}
\label{compare_sigma}
&
|\sigma_{f_k,\xi}(\theta)-\sigma_{f,\xi}(\theta)| 
\lesssim_{\gamma}
(k+1)^{\gamma/2}\|f\|_{C^{k+1,\gamma}}(1\vee \|\theta\|)^{\gamma}
\|\Sigma\|^{1/2}
{\mathbb E}^{1/2}\|\xi\|^2.
\end{align}
\end{lemma}

\begin{proof}
Note that 
\begin{align*}
&
|\sigma_{f_k,\xi}(\theta)-\sigma_{f,\xi}(\theta)| \leq 
|\sigma_{f_k-f,\xi}(\theta)|
\leq \sum_{j=1}^k {\mathbb E}^{1/2} \Bigl|({\mathcal B}^j f)'(\theta)(\xi)\Bigr|^2
\\
&
=\sum_{j=1}^k 
\Bigl\langle \Sigma ({\mathcal B}^j f)'(\theta), ({\mathcal B}^j f)'(\theta)\Bigr\rangle^{1/2}
\leq \|\Sigma\|^{1/2}\sum_{j=1}^k \|({\mathcal B}^j f)'(\theta)\|.
\end{align*}
Using bound \eqref{B^jf'L}, we get 
\begin{align*} 
&
\nonumber
|\sigma_{f_k,\xi}(\theta)-\sigma_{f,\xi}(\theta)| 
\\
&
\nonumber
\leq 
2^{\gamma} \|\Sigma\|^{1/2}(1\vee \|\theta\|)^{\gamma}
\sum_{j=1}^k\|f^{(j+1)}\|_{L_{\infty,\gamma}} 
(1+ j^{\gamma/2}{\mathbb E}^{1/2}\|\xi\|^{2\gamma})
({\mathbb E}^{1/2}\|\xi\|^2)^j
\\
&
\nonumber
\leq 2^{\gamma} \|f\|_{C^{k+1,\gamma}}\|\Sigma\|^{1/2}(1\vee \|\theta\|)^{\gamma}
(1+ k^{\gamma/2}{\mathbb E}^{1/2}\|\xi\|^{2\gamma})
\sum_{j=1}^k ({\mathbb E}^{1/2}\|\xi\|^2)^j
\\
&
\nonumber
\leq 2^{\gamma+1} \|f\|_{C^{k+1,\gamma}}\|\Sigma\|^{1/2}(1\vee \|\theta\|)^{\gamma}
(1+ k^{\gamma/2}{\mathbb E}^{1/2}\|\xi\|^{2\gamma})
{\mathbb E}^{1/2}\|\xi\|^2
\\
&
\lesssim_{\gamma}
(k+1)^{\gamma/2}\|f\|_{C^{k+1,\gamma}}(1\vee \|\theta\|)^{\gamma}
\|\Sigma\|^{1/2}
{\mathbb E}^{1/2}\|\xi\|^2.
\end{align*}
\qed
\end{proof}

Bound \eqref{bd_R} follows from representation \eqref{represent_R}, Theorem \ref{bias_control}, bound \eqref{conc_bound_f_k} and bound 
\eqref{compare_sigma}.

We now prove bound \eqref{psi_remainder}.
We can easily deduce from \eqref{conc_bound_f_k} that:
\begin{align*}
&
\nonumber
|S_{f_{k}}(\theta;\xi)-{\mathbb E}S_{f_{k}}(\theta;\xi)|
\lesssim_{\gamma} A_1 t^{1/2} \vee A_2 t^{(1+\gamma)/2} \vee A_3 t^{(1+\rho)/2}
\vee A_4 t^{(1+\rho+\gamma)/2} \vee A_5 t \vee A_6 t^{(2+\gamma)/2},
\end{align*}
where 
\begin{align*}
&
A_1\asymp_{\gamma} (k+1)^{\gamma/2} \|f\|_{C^{s,\gamma}} 
(1\vee \|\theta\|)^{\gamma} ({\mathbb E}\|\xi\|)^{\rho}\|\Sigma\|^{1/2},
\\
&
A_2\asymp_{\gamma}(k+1)^{\gamma/2} \|f\|_{C^{s,\gamma}} 
({\mathbb E}\|\xi\|)^{\rho}\|\Sigma\|^{(1+\gamma)/2},
\\
&
A_3 \asymp_{\gamma} (k+1)^{\gamma/2} \|f\|_{C^{s,\gamma}} 
(1\vee \|\theta\|)^{\gamma} 
\|\Sigma\|^{(1+\rho)/2},
\\
&
A_4\asymp_{\gamma}(k+1)^{\gamma/2} \|f\|_{C^{s,\gamma}} 
\|\Sigma\|^{(1+\rho+\gamma)/2},
\\
&
A_5 \asymp_{\gamma}  (k+1)^{\gamma/2} \|f\|_{C^{s,\gamma}} 
(1\vee \|\theta\|)^{\gamma} \|\Sigma\|,
\\
&
A_6\asymp_{\gamma} (k+1)^{\gamma/2} \|f\|_{C^{s,\gamma}} 
\|\Sigma\|^{(2+\gamma)/2}. 
\end{align*}
Using Lemma \ref{Orlicz}, we conclude that, for any $\psi$ 
satisfying the condition $\psi(t)\leq c_1 e^{c_2 t^{2/(2+\gamma)}}, t\geq 0,$ 
we have
\begin{align*}
\Bigl\|S_{f_{k}}(\theta;\xi)-{\mathbb E}S_{f_{k}}(\theta;\xi)\Bigr\|_{\psi}
\lesssim_{\gamma,\psi} A_1\vee \dots \vee A_m.
\end{align*}
Using the fact that $\|\Sigma\|\leq {\mathbb E}\|\xi\|^2\leq 1,$ it is easy to check that
\begin{align*}
A_1\vee \dots \vee A_m\lesssim_{\gamma} (k+1)^{\gamma/2} \|f\|_{C^{s,\gamma}} 
(1\vee \|\theta\|)^{\gamma} ({\mathbb E}^{1/2}\|\xi\|^2)^{\rho}\|\Sigma\|^{1/2}.
\end{align*}
Thus,
\begin{align}
\label{psi_norm}
\Bigl\|S_{f_{k}}(\theta;\xi)-{\mathbb E}S_{f_{k}}(\theta;\xi)\Bigr\|_{\psi}
\lesssim_{\gamma,\psi}
 (k+1)^{\gamma/2} \|f\|_{C^{s,\gamma}} 
(1\vee \|\theta\|)^{\gamma} ({\mathbb E}^{1/2}\|\xi\|^2)^{\rho}\|\Sigma\|^{1/2}.
\end{align}
Using bound \eqref{compare_sigma}, we get
\begin{align*}
&
\Bigl\|(\sigma_{f_k,\xi}(\theta)-\sigma_{f,\xi}(\theta))Z\Bigr\|_{\psi} 
\lesssim_{\gamma}
(k+1)^{\gamma/2}\|f\|_{C^{k+1,\gamma}}(1\vee \|\theta\|)^{\gamma}
{\mathbb E}^{1/2}\|\xi\|^2 \|\Sigma\|^{1/2}\|Z\|_{\psi},
\end{align*}
which is dominated by the right hand side of \eqref{psi_norm}.
Thus, we can conclude that 
\begin{align*}
\|R\|_{\psi} \lesssim_{\gamma,\psi}
(k+1)^{\gamma/2} \|f\|_{C^{s,\gamma}} 
(1\vee \|\theta\|)^{\gamma} \Bigl(({\mathbb E}^{1/2}\|\xi\|^2)^{\rho}\|\Sigma\|^{1/2}
\bigvee ({\mathbb E}\|\xi\|)^s
\Bigr),
\end{align*}
implying bound \eqref{psi_remainder}.

To prove normal approximation bound \eqref{norm_appr}, we need the following elementary lemma.

\begin{lemma}
\label{eta_Z}
For random variables $\eta_1, \eta_2,$ 
denote 
$$
\Delta (\eta_1, \eta_2):=\sup_{x\in {\mathbb R}}
|{\mathbb P}\{\eta_1 \leq x\}-{\mathbb P}\{\eta_2\leq x\}| 
$$
and 
$$
\delta (\eta_1, \eta_2):= \inf_{\delta>0}\Bigl[{\mathbb P}\{|\eta_1-\eta_2|\geq \delta\}+\delta\Bigr].
$$
Then, for an arbitrary random variable $\eta$ and a standard normal random variable $Z,$
$$
\Delta (\eta,Z) \leq \delta(\eta;Z).
$$
\end{lemma}

We apply this lemma to random variable $\eta:=\frac{f_k(X) - f(\theta)}{\sigma_{f,\xi}(\theta)}.$ Using representation \eqref{Z_R} and bound \eqref{bd_R}, we get
that, for all $t\geq 1$ with probability at least $1-e^{-t}$
\begin{align*}
&
\biggl|\frac{f_k(X) - f(\theta)}{\sigma_{f,\xi}(\theta)}-Z\biggr|
\lesssim_{\gamma} 
(k+1)^{\gamma/2} K_{s,\gamma}(f;\Sigma;\theta) (1\vee \|\Sigma\|^{1/2}\sqrt{t})^{\gamma}
\\
&
\biggl(
({\mathbb E}^{1/2}\|\xi\|^2)^{\rho}\sqrt{t}
\bigvee \|\Sigma\|^{\rho/2}t^{(1+\rho)/2}\bigvee \|\Sigma\|^{1/2} t
\bigvee \frac{({\mathbb E}^{1/2}\|\xi\|^2)^{s}}{\|\Sigma\|^{1/2}}
\biggr). 
\end{align*}
Let $t:= \log\biggl(\frac{1}{\|\Sigma\|}\biggr).$ With this choice of $t,$
it is easy to see that 
$$
\|\Sigma\|^{1/2}\sqrt{t}\lesssim 1\ {\rm and}\ \|\Sigma\|^{1/2}t\lesssim  
\|\Sigma\|^{\rho/2}t^{(1+\rho)/2}.
$$
Thus, with probability at least $1-\|\Sigma\|,$
\begin{align*}
&
\biggl|\frac{f_k(X) - f(\theta)}{\sigma_{f,\xi}(\theta)}-Z\biggr|
\lesssim_{\gamma} 
(k+1)^{\gamma/2} K_{s,\gamma}(f;\Sigma;\theta)
\\
&
\biggl(
({\mathbb E}^{1/2}\|\xi\|^2)^{\rho}\sqrt{\log\biggl(\frac{1}{\|\Sigma\|}\biggr)}
\bigvee \|\Sigma\|^{\rho/2}\log^{(1+\rho)/2}\biggl(\frac{1}{\|\Sigma\|}\biggr)
\bigvee \frac{({\mathbb E}^{1/2}\|\xi\|^2)^{s}}{\|\Sigma\|^{1/2}}
\biggr). 
\end{align*}
It follows from Lemma \ref{eta_Z} that 
\begin{align*}
&
\Delta(\eta;Z)\leq \delta(\eta;Z) 
\lesssim_{\gamma} 
(k+1)^{\gamma/2} 
K_{s,\gamma}(f;\Sigma;\theta)
\\
&
\biggl(
({\mathbb E}^{1/2}\|\xi\|^2)^{\rho}\sqrt{\log\biggl(\frac{1}{\|\Sigma\|}\biggr)}
\bigvee \|\Sigma\|^{\rho/2}\log^{(1+\rho)/2}\biggl(\frac{1}{\|\Sigma\|}\biggr)
\bigvee \frac{({\mathbb E}^{1/2}\|\xi\|^2)^{s}}{\|\Sigma\|^{1/2}}
\biggr) + \|\Sigma\|. 
\end{align*}
Since also 
$$
\|\Sigma\|\leq \|\Sigma\|^{\rho/2}\log^{(1+\rho)/2}\biggl(\frac{1}{\|\Sigma\|}\biggr),
$$
we can conclude that 
\begin{align*}
&
\Delta(\eta;Z)
\lesssim_{\gamma} 
(k+1)^{\gamma/2} K_{s,\gamma}(f;\Sigma;\theta)
\\
&
\biggl(
({\mathbb E}^{1/2}\|\xi\|^2)^{\rho}\sqrt{\log\biggl(\frac{1}{\|\Sigma\|}\biggr)}
\bigvee \|\Sigma\|^{\rho/2}\log^{(1+\rho)/2}\biggl(\frac{1}{\|\Sigma\|}\biggr)
\bigvee \frac{({\mathbb E}^{1/2}\|\xi\|^2)^{s}}{\|\Sigma\|^{1/2}}
\biggr).
\end{align*}

\qed
\end{proof}


\section{The proof of efficiency: a lower bound}\label{Sec:Lowerbounds}

Our goal in this section is to prove Theorem \ref{min_lower_bd}.  
It will be convenient for our purposes to represent the noise as a sum of a series with 
i.i.d. standard normal coefficients. To this end, we use the following well known result.

\begin{theorem}[\cite{Kwapien}]
\label{theorem:Kwapien_Szymanski}
Let $\xi\in \EB$, $\xi \sim \mathcal{N}(0;\Sigma).$ There exists a sequence 
$\{g_{k}\}_{k \in \mathbb{N}}$ of i.i.d. standard normal random variables and a sequence 
$\{x_{k}\}_{k \in \mathbb{N}}$ in $\EB$ such that, for all $k\in \mathbb{N},$ 
$x_k\not\in \overline{\spann\{x_j:j\neq k\}},$
$
\xi = \sum_{k=1}^{\infty} x_{k}g_{k}
$
with the series in the right hand side converging in $\EB$ a.s., and 
$
\sum_{k=1}^{\infty} \|x_{k}\|^{2} <\infty.
$
\end{theorem}

Clearly, $\overline{{\rm Im}(\Sigma)}=\overline{\spann\{x_j:j\in {\mathbb N}\}}.$
In the rest of this section, we provide the proof of Theorem \ref{min_lower_bd}. 

\begin{proof}
First, we will replace $\sigma^{2}_{f,\xi}(\theta)$ in the lower bound with $\sigma^{2}_{f,\xi}(\theta_0).$
To this end, we use the following simple lemma.

\begin{lemma}
\label{theta>>>theta_0}
For all $\theta\in E$ such that 
$$
\|\theta-\theta_0\|\leq c\|\Sigma\|^{1/2}< 1,
$$
the following bound holds:
\begin{align*}
\biggl|\frac{\sigma^{2}_{f,\xi}(\theta)}{\sigma^{2}_{f,\xi}(\theta_0)}-1\biggr|
\leq 2^{s+2\gamma} K^{2}_{s,\gamma}(f;\Sigma;\theta_0) c^{s-1}\|\Sigma\|^{(s-1)/2}.
\end{align*}
\end{lemma}

\begin{proof}
Here and in what follows, denote $\rho:=s-1.$
We have 
\begin{align*}
&
\biggl|\frac{\sigma^{2}_{f,\xi}(\theta)}{\sigma^{2}_{f,\xi}(\theta_0)}-1\biggr|
=
\frac{\Bigl| \langle \Sigma f^{\prime}(\theta),f^{\prime}(\theta)\rangle-
\langle \Sigma f^{\prime}(\theta_0),f^{\prime}(\theta_0)\rangle
\Bigr|}{\sigma^{2}_{f,\xi}(\theta_0)}
\\
&
\leq
\frac{\|\Sigma\|\|f^{\prime}(\theta)-f^{\prime}(\theta_0)\|
(\|f^{\prime}(\theta)\|+\|f^{\prime}(\theta_0)\|)}{\sigma^{2}_{f,\xi}(\theta_0)} 
\\
&
\leq 
\frac{\|\Sigma\|\|f^{\prime}\|_{{\rm Lip}_{\rho,\gamma}}(1\vee \|\theta\|\vee \|\theta_0\|)^{\gamma}\|\theta-\theta_0\|^{\rho}
\|f^{\prime}\|_{L_{\infty,\gamma}}\Bigl((1\vee \|\theta\|)^{\gamma}+(1\vee \|\theta_0\|)^{\gamma}\Bigr)}{\sigma^{2}_{f,\xi}(\theta_0)}
\end{align*}
We then use the condition $\|\theta-\theta_0\|\leq 1$ to get 
$$
\|\theta\|\leq \|\theta_0\| + \|\theta-\theta_0\|\leq 2(1\vee \|\theta_0\|).
$$
Therefore, 
\begin{align*}
&
\biggl|\frac{\sigma^{2}_{f,\xi}(\theta)}{\sigma^{2}_{f,\xi}(\theta_0)}-1\biggr|
\leq 
\frac{2^{2\gamma+1}\|\Sigma\|\|f\|_{C^{s,\gamma}}^2(1\vee \|\theta_0\|)^{2\gamma}\|\theta-\theta_0\|^{\rho}}{\sigma^{2}_{f,\xi}(\theta_0)}
\\
&
\leq 2^{2\gamma+1} K^2(f;\Sigma;\theta_0) \|\theta-\theta_0\|^{\rho} 
\leq 2^{2\gamma+1+\rho} K^{2}(f;\Sigma;\theta_0) c^{\rho}\|\Sigma\|^{\rho/2},
\end{align*}
concluding the proof.
\QED
\end{proof}

The bound of Lemma \ref{theta>>>theta_0} implies that
\begin{align}
\label{theta_theta_0}
&
\nonumber
\sup_{\|\theta-\theta_0\|\leq c\|\Sigma\|^{1/2}}
\frac{\E_{\theta} (T(X)-f(\theta))^2}{\sigma^{2}_{f,\xi}(\theta)} 
=
\sup_{\|\theta-\theta_0\|\leq c\|\Sigma\|^{1/2}}
\frac{\E_{\theta} (T(X)-f(\theta))^2}{\sigma^{2}_{f,\xi}(\theta_0)}
\frac{\sigma^{2}_{f,\xi}(\theta_0)}{\sigma^{2}_{f,\xi}(\theta)}
\\
&
\nonumber
\geq 
\sup_{\|\theta-\theta_0\|\leq c\|\Sigma\|^{1/2}}
\frac{\E_{\theta} (T(X)-f(\theta))^2}{\sigma^{2}_{f,\xi}(\theta_0)}
\frac{1}{1+ \sup_{\|\theta-\theta_0\|\leq c\|\Sigma\|^{1/2}}\biggl|\frac{\sigma^{2}_{f,\xi}(\theta)}{\sigma^{2}_{f,\xi}(\theta_0)}-1\biggr|} 
\\
&
\geq 
\sup_{\|\theta-\theta_0\|\leq c\|\Sigma\|^{1/2}}
\frac{\E_{\theta} (T(X)-f(\theta))^2}{\sigma^{2}_{f,\xi}(\theta_0)}
\frac{1}{1+2^{s+2\gamma} K^{2}_{s,\gamma}(f;\Sigma;\theta_0) c^{s-1}\|\Sigma\|^{(s-1)/2}}. 
\end{align}

The rest of the proof is based on a finite-dimensional approximation and an application 
of van Trees inequality. 
For a fixed $N\in \mathbb{N},$ let 
\begin{equation*}
\label{def:LN*}
L_{N}\coloneqq \spann\{x_{1},\dots,x_{N}\}\subset \EB, 
\end{equation*}
and 
\begin{equation}
\label{def:xiN}
\xi_{N}\coloneqq \sum_{k=1}^{N} x_{k}g_{k}\in L_{N}, \
\xi_{N}^{\perp}\coloneqq \xi-\xi_{N}=\sum_{k>N}x_{k}g_{k}. 
\end{equation}
Clearly, random variables $\xi_N$ and $\xi_N^{\perp}$ are independent.

We define a linear mapping $A_{N}: \R^{N}\mapsto L_{N}$ such that, for all $(\alpha_{1},\dots,\alpha_{N})\in \R^{N},$
$
A_{N}(\alpha_{1},\dots,\alpha_{N})\coloneqq \sum_{k=1}^{N}\alpha_{k}x_{k}.
$
 Since $x_1,\dots, x_N$ are linearly independent vectors and $L_N$ is an $N$-dimensional 
 subspace of $E,$ $A_N$ is a bijection between the spaces $\R^{N}$ and $L_N$ with 
 inverse $A_N^{-1}: L_N\mapsto \R^N.$ 
 In what follows, $\R^N$ is viewed as a Euclidean space with canonical inner 
 product. Denote by $L_N^{*}\supset E^*$ the dual space of $L_N$ and let $A_N^{*}: L_N^{*}\mapsto \R^N$
 be the adjoint operator of $A_N.$  For $\alpha =(\alpha_1,\dots, \alpha_N)\in \R^N$
 and $u\in L_N^{*},$ we have 
 $$
 \langle \alpha, A_N^{*}u\rangle = \langle A_N \alpha,u\rangle
 =\sum_{j=1}^N \alpha_j \langle x_j,u\rangle,
 $$ 
 implying that 
 $
 A_N^{*}u=\Bigl(\langle x_j,u\rangle : j=1,\dots, N\Bigr).
 $
With some abuse of notation, we denote by $\langle \cdot, \cdot \rangle$ both 
the inner product of $\R^N$ (and other inner product spaces) and the action of a linear 
functional on a vector in a Banach space.  

Let $Z_{N}\coloneqq (g_{1},\dots,g_{N})\sim \mathcal{N}(0,I_{N}).$ Then  
$\xi_{N}=A_{N}Z_{N}.$ Denote by $\Sigma_N$ the covariance operator 
of $\xi_N:$
$
\Sigma_N u:= \E\langle \xi_N,u\rangle \xi_N, u\in L_N^*.
$
Then 
$$
\Sigma_N u = \sum_{j=1}^N \langle x_j,u\rangle x_j= A_N A_N^*u, u\in L_N^*, 
$$
implying that 
\begin{equation}
\label{def:sigma_xiN}
\Sigma_{N}=A_{N}A_{N}^{*}.
\end{equation} 
It is easy to check that 
$
\|\Sigma_N-\Sigma\|\to 0\ {\rm as}\ N\to\infty,
$
which follows from the bound 
$$
\|(\Sigma-\Sigma_N)u\|= \biggl\|\sum_{j>N+1} \langle x_j,u\rangle x_j\biggr\|
\leq \sum_{j>N+1} \|x_j\|^2 \|u\|, u\in E^*
$$
and the condition $\sum_{j\in {\mathbb N}}\|x_j\|^2<\infty.$ 
It is also easy to see that, for all $u\in E^*,$  $\langle \Sigma_N u,u\rangle$
monotonically converges to $\langle \Sigma u,u\rangle$
and that $\|\Sigma_N\|\leq \|\Sigma\|, N\geq 1.$

Since $\theta_0\in \overline{\spann\{x_j:j\in {\mathbb N}\}},$ there exists a sequence $\theta_{0,N}\in L_N$ such that $\theta_{0,N}\to \theta_0$ as $N\to\infty.$ Therefore,
$$
\sigma^{2}_{f,\xi_N}(\theta_{0,N}) = 
\langle \Sigma_N f^{\prime}(\theta_{0,N}), f^{\prime}(\theta_{0,N})\rangle
\to \langle \Sigma f^{\prime}(\theta_{0}), f^{\prime}(\theta_{0})\rangle
=\sigma^{2}_{f,\xi}(\theta_{0})\ {\rm as}\ N\to\infty.
$$
By a simple continuity argument, it also follows that 
$$
K(f;\Sigma_N;\theta_{0,N})\to K(f;\Sigma;\theta_0)\ {\rm as}\ N\to\infty.
$$
Thus, for all large enough $N,$
\begin{align*}
&
U(\theta_0;c;\Sigma):=
\Bigl\{\theta\in E: \|\theta-\theta_0\|\leq c\|\Sigma\|^{1/2}\Bigr\}
\\
&
\supset 
\Bigl\{\theta\in L_N: \|\theta-\theta_{0,N}\|\leq \frac{c}{2}\|\Sigma_N\|^{1/2}\Bigr\}
=:U_N(\theta_{0,N};c/2;\Sigma_N).
\end{align*}
Using a simple conditioning argument and Jensen's inequality, this implies 
\begin{align*}
&
\sup_{\theta \in U(\theta_0;c;\Sigma)}
\E_{\theta} (T(X)-f(\theta))^2 
\geq  
\sup_{\theta \in U_N(\theta_{0,N};c/2;\Sigma_N)}
\E_{\theta} (T(X)-f(\theta))^2 
\\
&
\geq 
\sup_{\theta\in U_N(\theta_{0,N};c/2;\Sigma_N)} 
\E\E\{(T(\theta+\xi_{N}+ \xi_{N}^{\perp})-f(\theta))^{2}\vert \xi_N\}
\\
&
\geq 
\sup_{\theta\in U_N(\theta_{0,N};c/2;\Sigma_N)} 
\E(\E\{T(\theta+\xi_{N}+ \xi_{N}^{\perp})\vert \xi_{N} \}-f(\theta))^{2}
\\&=
\sup_{\theta\in U_N(\theta_{0,N};c/2;\Sigma_N)} \E_{\theta}(\tilde{T}(X_{N})-f(\theta))^{2},
\end{align*}
where
\begin{align*}
X_{N}\coloneqq \theta+\xi_{N} \in L_{N}\text{ and } 
 \tilde{T}(x) \coloneqq \E T(x+\xi_N^{\perp}), x\in E. 
\end{align*}
Next, we get 
\begin{align}
\label{U_U_N}
&
\nonumber
\sup_{\theta \in U(\theta_0;c;\Sigma)}
\frac{\E_{\theta} (T(X)-f(\theta))^2}{\sigma^{2}_{f,\xi}(\theta_0)}
\\
&
\geq 
\sup_{\theta\in U_N(\theta_{0,N};c/2;\Sigma_N)}\frac{\E_{\theta}(\tilde{T}(X_{N})-f(\theta))^{2}}
{\sigma^{2}_{f,\xi_N}(\theta_{0,N})}
\frac{\sigma^{2}_{f,\xi_N}(\theta_{0,N})}{\sigma^{2}_{f,\xi}(\theta_0)}.
\end{align}
To bound 
$$
\sup_{\theta\in U_N(\theta_{0,N};c/2;\Sigma_N)}\frac{\E_{\theta}(\tilde{T}(X_{N})-f(\theta))^{2}}
{\sigma^{2}_{f,\xi_N}(\theta_{0,N})}
$$
from below, we will use the following lemma whose proof is based on an application 
of van Trees inequality (see \cite{GillLevit}).

\begin{lemma}
\label{apply_van_Trees}
Under the assumptions of Theorem \ref{min_lower_bd}, for some constant $D_{\gamma}'>0$ 
and for all large enough $N,$ the following bound holds for an arbitrary estimator $T(X_N):$
\begin{align*}
&
\sup_{\theta\in U_N(\theta_{0,N};c/2;\Sigma_N)}\frac{\E_{\theta}(T(X_{N})-f(\theta))^{2}}
{\sigma^{2}_{f,\xi_N}(\theta_{0,N})}
\\
&
\geq 
1-D_{\gamma}' K^{2}_{s,\gamma}(f;\Sigma_N;\theta_{0,N}) \Bigl(c^{s-1}\|\Sigma_N\|^{(s-1)/2}+\frac{1}{c^2}\Bigr). 
\end{align*}
\end{lemma}

To complete the proof of Theorem \ref{min_lower_bd}, use bounds \eqref{theta_theta_0},
\eqref{U_U_N} and the bound of Lemma \ref{apply_van_Trees} to get 
\begin{align*}
&
\nonumber
\sup_{\|\theta-\theta_0\|\leq c\|\Sigma\|^{1/2}}
\frac{\E_{\theta} (T(X)-f(\theta))^2}{\sigma^{2}_{f,\xi}(\theta)} 
\\
&
\nonumber
\geq 
\frac{1-D_{\gamma}' K_{s,\gamma}^{2}(f;\Sigma_N;\theta_{0,N}) \Bigl(c^{s-1}\|\Sigma_N\|^{(s-1)/2}+\frac{1}{c^2}\Bigr)}
{1+2^{s+2\gamma} K^{2}_{s,\gamma}(f;\Sigma;\theta_0) c^{s-1}\|\Sigma\|^{(s-1)/2}}
\frac{\sigma^{2}_{f,\xi_N}(\theta_{0,N})}{\sigma^{2}_{f,\xi}(\theta_0)}.
\end{align*}
Passing to the limit as $N\to\infty,$ we get
\begin{align*}
&
\nonumber
\sup_{\|\theta-\theta_0\|\leq c\|\Sigma\|^{1/2}}
\frac{\E_{\theta} (T(X)-f(\theta))^2}{\sigma^{2}_{f,\xi}(\theta)} 
\geq 
\frac{1-D_{\gamma}' K_{s,\gamma}^{2}(f;\Sigma;\theta_{0}) \Bigl(c^{s-1}\|\Sigma\|^{(s-1)/2}+\frac{1}{c^2}\Bigr)}
{1+2^{s+2\gamma} K^{2}_{s,\gamma}(f;\Sigma;\theta_0) c^{s-1}\|\Sigma\|^{(s-1)/2}}
\\
&
\nonumber
\geq 
1-(D_{\gamma}'+2^{s+2\gamma}) K_{s,\gamma}^{2}(f;\Sigma;\theta_{0}) \Bigl(c^{s-1}\|\Sigma\|^{(s-1)/2}+\frac{1}{c^2}\Bigr),
\end{align*}
implying the bound of Theorem \ref{min_lower_bd}.

\QED
\end{proof}

Finally, we prove Lemma \ref{apply_van_Trees}.

\begin{proof}
Let $c':= \frac{c}{K_{s,\gamma}(f;\Sigma_N;\theta_{0,N})}.$
For $t\in [-c'/2,c'/2],$ $\theta_{0,N}\in L_N$ and $h\in L_{N},$ define
\begin{equation*}
\theta_{t}\coloneqq \theta_{0,N}+ t h, \quad X_N\coloneqq  \theta_{t} + \xi_{N}.
\end{equation*}
Consider a problem of estimation of a function 
$$
\varphi(t)\coloneqq f(\theta_t), t\in [-c'/2,c'/2]
$$
based on an observation $X_N\sim {\mathcal N}(\theta_t,\Sigma_N), t\in [-c'/2,c'/2].$
Since $A_N: \R^N\mapsto L_N$ is a bijection, an equivalent problem is 
to estimate $\varphi(t)$ based on an observation 
\begin{equation*}
A_{N}^{-1}X\coloneqq A_{N}^{-1}\theta_{t} + Z_{N} \sim \mathcal{N}(A_{N}^{-1}\theta_{t};I_{N}).\end{equation*}
The Fisher information for the model $A_{N}^{-1}X \sim \mathcal{N}(A_{N}^{-1}\theta_{t};I_{N})$ with $t\in [-c'/2,c'/2]$ is equal to 
\begin{align*}
I(t)=I&= \inner{A_{N}^{-1}h}{A_{N}^{-1}h}.
\end{align*}
We will choose $h\coloneqq \frac{\Sigma_N f^{\prime}(\theta_{0,N})}{\sigma_{f,\xi_N}(\theta_{0,N})}.$
For this choice of $h,$ 
\begin{align*}
\frac{c'}{2}\|h\| &\leq 
\frac{(c'/2) \|\Sigma_N\| \|f^{\prime}(\theta_{0,N})\|}{\sigma_{f,\xi_N}(\theta_{0,N})}
\\
&
\leq 
\frac{c'}{2}\frac{\|f\|_{C^{s,\gamma}}(1\vee\|\theta_{0,N}\|)^{\gamma}\|\Sigma_N\|^{1/2}}{\sigma_{f,\xi_N}(\theta_{0,N})} \|\Sigma_N\|^{1/2}
\\
&
=
\frac{c'}{2}K_{s,\gamma}(f;\Sigma_N;\theta_{0,N})\|\Sigma_N\|^{1/2}=\frac{c}{2}\|\Sigma_N\|^{1/2}<1,
\end{align*}
implying that, for all large enough $N,$ $\theta_t\in U_N(\theta_{0,N};c/2;\Sigma_N), |t|\leq c'/2$
and, as a consequence,
\begin{align}
\label{>vanTrees}
&
\nonumber
\sup_{\theta\in U_N(\theta_{0,N};c/2;\Sigma_N)}\frac{\E_{\theta}(T(X_{N})-f(\theta))^{2}}
{\sigma^{2}_{f,\xi_N}(\theta_{0,N})}
\geq 
\sup_{t\in [-c'/2,c'/2]} \frac{\E_{t} (T(X_{N})-\varphi(t))^{2}}{\sigma^{2}_{f,\xi_N}(\theta_{0,N})}
\\
&
=
\sup_{t\in [-c'/2,c'/2]} \frac{\E_{t} (\hat T(A_N^{-1}X_{N})-\varphi(t))^{2}}{\sigma^{2}_{f,\xi_N}(\theta_{0,N})},
\end{align}
where $\hat T(x):=T(A_N x), x\in \R^N.$
We also have 
\begin{align*}
I &= \frac{\inner{A_{N}^{-1}\Sigma_N f^{\prime}(\theta_{0,N})}{A_{N}^{-1}\Sigma_N f^{\prime}(\theta_{0,N})}}{\sigma_{f,\xi_N}^2(\theta_{0,N})}
=\frac{\inner{A_{N}^{-1}A_N A_N^{*}f^{\prime}(\theta_{0,N})}{A_{N}^{-1}A_N A_N^{*} f^{\prime}(\theta_{0,N})}}{\sigma_{f,\xi_N}^2(\theta_{0,N})}
\\
&
=\frac{\inner{A_N^{*}f^{\prime}(\theta_{0,N})}{A_N^{*} f^{\prime}(\theta_{0,N})}}{\sigma_{f,\xi_N}^2(\theta_{0,N})}
=
\frac{\inner{A_N A_N^{*}f^{\prime}(\theta_{0,N})}{f^{\prime}(\theta_{0,N})}}{\sigma_{f,\xi_N}^2(\theta_{0,N})}
=\frac{\inner{\Sigma_N f^{\prime}(\theta_{0,N})}{f^{\prime}(\theta_{0,N})}}{\sigma_{f,\xi_N}^2(\theta_{0,N})}=1.
\end{align*}
Let $\pi$ be a prior density on $[-1,1]$ with $\pi(-1)=\pi(1)=0$ and such that 
$$
J_{\pi}:= \int_{-1}^1 \frac{(\pi^{\prime}(s))^2}{\pi(s)}ds <\infty.
$$
Denote $\pi_{c'}(t):= \frac{2}{c'}\pi\Bigl(\frac{2t}{c'}\Bigr), t\in [-c'/2,c'/2].$
Then $J_{\pi_{c'}}= \frac{4J_{\pi}}{(c')^2}.$

By van Trees inequality, for any estimator $\hat T(A_{N}^{-1}X_N)$ of $\varphi(t),$ it holds that 
\begin{align}
&
\nonumber
\sup_{t\in [-c'/2,c'/2]} \E_{t} (\hat T(A_N^{-1}X_{N})-\varphi(t))^{2}
\geq 
\int_{-c'/2}^{c'/2} \E_{t} (\hat T(A_N^{-1}X_{N}) -\varphi(t))^{2} \pi_{c'}(t)dt
\geq
\\
&
\geq \frac{\bigl(
\int_{-c'/2}^{c'/2}\varphi^{\prime}(t)\pi_{c'}(t)dt
\bigr)^{2}}
{
\int_{-c'/2}^{c'/2} I(t)dt+4J_{\pi}/(c')^{2}
}
\geq
\frac
{\bigl(
\int_{-c'/2}^{c'/2}\varphi^{\prime}(t)\pi_{c'}(t)dt
\bigr)^{2}}
{
1+4J_{\pi}/(c')^{2}
}.
\label{ineq:lowerdecomp}
\end{align}
It remains to bound from below 
$
\bigl(
\int_{-c'/2}^{c'/2}\varphi^{\prime}(t)\pi_{c'}(t)dt
\bigr)^{2}.
$
Note that
$
\varphi^{\prime}(t)=\inner{h}{f^{\prime}(\theta_{t})}
$
and let 
\begin{align*}
I_{0}&= \int_{-c'/2}^{c'/2}
\inner{h}{f^{\prime}(\theta_{0,N})}\pi_{c'}(t)dt
=\inner{h}{f^{\prime}(\theta_{0,N})},
\\
I_{1}&= \int_{-c'/2}^{c'/2}[\varphi^{\prime}(t)-\varphi^{\prime}(0)]\pi_{c'}(t)dt.
\end{align*}
We have 
\begin{align*}
&
\biggl(
\int_{-c'/2}^{c'/2}\varphi^{\prime}(t)\pi_{c'}(t)dt
\biggr)^{2}
\\
&
=(I_0 +I_1)^2
\geq I_0^2 -2|I_0||I_1|
\geq 
\inner{h}{f^{\prime}(\theta_{0,N})}^2 
-2|\inner{h}{f^{\prime}(\theta_{0,N})}||I_1|.
\end{align*}
With $h=\frac{\Sigma_N f^{\prime}(\theta_{0,N})}{\sigma_{f,\xi_N}(\theta_{0,N})},$
we get
\begin{align*}
\inner{h}{f^{\prime}(\theta_{0,N})}^2 
= \frac{\inner{\Sigma_Nf^{\prime}(\theta_{0,N})}{f^{\prime}(\theta_{0,N})}^2}{\sigma_{f,\xi_N}^2(\theta_{0,N})} = \sigma_{f,\xi_N}^2(\theta_{0,N})
\end{align*}
and 
\begin{align}
\label{low_bd_int_varphi}
&
\biggl(
\int_{-c'/2}^{c'/2}\varphi^{\prime}(t)\pi_{c'}(t)dt
\biggr)^{2}
\geq 
\sigma_{f,\xi_N}^2(\theta_{0,N})-2\sigma_{f,\xi_N}(\theta_{0,N})|I_1|.
\end{align}
Finally, we bound $|I_1|$ as follows.
Note that 
\begin{align*}
|\varphi^{\prime}(t)-\varphi^{\prime}(0)|
&=|\inner{h}{f^{\prime}(\theta_{t})- f^{\prime}(\theta_{0,N})}|
\\&\leq
 \|h\| \|f^{\prime}\|_{{\rm Lip}_{\rho,\gamma}}
 (1 \lor \|\theta_{0,N}\|\lor \|\theta_t\|)^{\gamma}
(c'/2)^{\rho}\|h\|^{\rho}
\\&\leq
 \|f\|_{C^{s,\gamma}}
 (1 \lor \|\theta_{0,N}\|\lor (\|\theta_{0,N}\|+(c'/2)\|h\|))^{\gamma}
(c'/2)^{\rho}\|h\|^{1+\rho}
\\&\leq
2^{\gamma-\rho}\|f\|_{C^{s,\gamma}}
 (1 \lor \|\theta_{0,N}\|)^{\gamma}
(c')^{\rho}\|h\|^{1+\rho},
\end{align*}
where we used the fact that $(c'/2)\|h\|\leq 1.$
It follows that 
\begin{align*}
 |I_{1}|&\leq  
2^{\gamma-\rho}\|f\|_{C^{s,\gamma}}
 (1 \lor \|\theta_{0,N}\|)^{\gamma}
(c')^{\rho}\|h\|^{1+\rho}
\\&\leq
2^{\gamma-\rho}
\frac{\|f\|_{C^{s,\gamma}}
 (1 \lor \|\theta_{0,N}\|)^{\gamma}}
 {\sigma_{f,\xi_N}^{1+\rho}(\theta_{0,N})}
(c')^{\rho}\|\Sigma_N\|^{1+\rho}\|f^{\prime}(\theta_{0,N})\|^{1+\rho}
\\&\leq
2^{\gamma-\rho}\sigma_{f,\xi_N}(\theta_{0,N})
\frac{\|f\|_{C^{s,\gamma}}
 (1 \lor \|\theta_{0,N}\|)^{\gamma}}
 {\sigma_{f,\xi_N}^{2+\rho}(\theta_{0,N})}
(c')^{\rho}\|\Sigma_N\|^{1+\rho}
\|f^{\prime}\|_{L_{\infty,\gamma}}^{1+\rho}(1\vee \|\theta_{0,N}\|)^{\gamma (1+\rho)}
\\&\leq
2^{\gamma-\rho}\sigma_{f,\xi_N}(\theta_{0,N})
\frac{\|f\|_{C^{s,\gamma}}^{2+\rho}
 (1 \lor \|\theta_{0,N}\|)^{\gamma (2+\rho)}\|\Sigma_N\|^{(2+\rho)/2}}
 {\sigma_{f,\xi_N}^{2+\rho}(\theta_{0,N})}
(c')^{\rho}\|\Sigma_N\|^{\rho/2}
\\&=
2^{\gamma-\rho}\sigma_{f,\xi_N}(\theta_{0,N})
K_{s,\gamma}^{2+\rho}(f;\Sigma_N;\theta_{0,N})
(c')^{\rho}\|\Sigma_N\|^{\rho/2}.
\end{align*}
We substitute this bound in \eqref{low_bd_int_varphi} to get
\begin{align}
\label{low_bd_int_varphi_final}
&
\nonumber
\biggl(
\int_{-c'/2}^{c'/2}\varphi^{\prime}(t)\pi_{c'}(t)dt
\biggr)^{2}
\\
&
\geq 
\sigma_{f,\xi_N}^2(\theta_{0,N})\Bigl(1-2^{\gamma+1-\rho}
K_{s,\gamma}^{2+\rho}(f;\Sigma_N;\theta_{0,N})
(c')^{\rho}\|\Sigma_N\|^{\rho/2}\Bigr).
\end{align}
Using bounds \eqref{>vanTrees}, \eqref{ineq:lowerdecomp} and \eqref{low_bd_int_varphi_final},
we conclude that 
\begin{align*}
&
\sup_{\theta\in U_N(\theta_{0,N};c/2;\Sigma_N)}\frac{\E_{\theta}(T(X_{N})-f(\theta))^{2}}
{\sigma^{2}_{f,\xi_N}(\theta_{0,N})}
\geq 
\frac
{
1-2^{\gamma+1-\rho}
K_{s,\gamma}^{2+\rho}(f;\Sigma_N;\theta_{0,N})
(c')^{\rho}\|\Sigma_N\|^{\rho/2}
}
{
1+4J_{\pi}/(c')^{2}
}
\\
&
\geq
1-2^{\gamma+1-\rho}
K_{s,\gamma}^{2}(f;\Sigma_N;\theta_{0,N})
c^{\rho}\|\Sigma_N\|^{\rho/2}
-\frac{4J_{\pi}K_{s,\gamma}^{2}(f;\Sigma_N;\theta_{0,N})}{c^2},
\end{align*}
implying the claim of the lemma.
\QED
\end{proof}

\section{The proof of minimax lower bound}\label{Sec:Optimal_Smoothness}
\label{nemirovski}

In this section, we use a modification of the approach developed by Nemirovski \cite{Nemirovski_1990, Nemirovski} to prove minimax lower bounds 
implying the optimality of smoothness thresholds for efficient estimation. 
This will be done only in the case of classical Gaussian shift model 
(see Example \ref{Example 1})
$$
X= \theta + \sigma Z,\ \theta \in {\mathbb R}^d,\ Z\sim {\mathcal N}(0,I_d)
$$
with unknown mean $\theta$ and known noise level $\sigma^2.$ 
The noise in this model is $\xi:= \sigma Z$ with covariance $\Sigma=\sigma^2 I_d,$ and the parameter 
space is the Euclidean space ${\mathbb R}^d$ with canonical 
inner product. Our main goal is to prove Theorem \ref{min_max_nemirovski} stated 
in Section \ref{Overview}. Our approach is based on a construction of 
a set $\Theta$ of $2^{d/8}$ $2\eps$-separated points of the unit ball in ${\mathbb R}^d$ and a set of 
smooth functionals $f_l(\theta), l=1,\dots , d.$ Assuming the existence of estimators 
$T_l(X), l=1,\dots, l$ of these functionals with mean squared error rate $\delta^2,$ we show 
that it is possible to estimate parameter $\theta\in \Theta$ with mean squared error 
$\lesssim \frac{\delta^2}{\eps^{2(s-1)}}.$ We compare this with well known minimax rates
of estimation of $\theta\in \Theta$ to prove a lower bound on $\delta^2.$ 

\begin{proof}
Let $h$ be the Hamming distance on the binary cube $\{-1,1\}^d:$
$$
h(\omega,\omega'):= \sum_{j=1}^d I(\omega_j\neq \omega_j'), \omega, \omega'\in \{-1,1\}^d.
$$
It follows from Varshamov-Gilbert bound (see \cite{Tsybakov}, Lemma 2.9)
that there exists a subset $\Omega\subset \{-1,1\}^d$ such that ${\rm card}(\Omega)\geq 2^{d/8}$
and $h(\omega, \omega')\geq d/8, \omega\neq \omega', \omega,\omega'\in  \Omega.$ 
For some $\eps \in (0,1/8),$ let
$$
\theta_{\omega}:= \frac{8\eps}{\sqrt{d}}(\omega_1,\dots, \omega_d), \omega\in \Omega
$$
and let $\Theta:= \{\theta_{\omega}: \omega \in \Omega\}.$
Note that $\|\theta_{\omega}\|=8\eps$ and 
\begin{align}
\label{dist_theta_omega}
\|\theta_{\omega}-\theta_{\omega'}\|= 16\eps \sqrt{\frac{h(\omega,\omega')}{d}},
\omega, \omega'\in \Omega,
\end{align}
which implies that, for all $\omega\neq \omega',$ 
$$
\|\theta_{\omega}-\theta_{\omega'}\|\geq \frac{8}{\sqrt{2}}\eps\geq 
2\eps.
$$

Let $\varphi : {\mathbb R} \mapsto [0,1]$ be a $C^{\infty}$ function 
with support in $[-1,1],$ with $\|\tilde \varphi\|_{C^s}\leq 1$ for 
$\tilde \varphi (t):=\varphi(\|t\|^2), t\in {\mathbb R}^d$ and $\varphi(0)>0$ being 
a constant.  
Define
$$
f_l(\theta):= \sum_{\omega\in \Omega} \omega_l 
\eps^s \tilde \varphi\biggl(\frac{\theta-\theta_\omega}{\eps}\biggr), \theta \in {\mathbb R}^d.
$$
Note that the functions $\eps^s \tilde \varphi\Bigl(\frac{\theta-\theta_\omega}{\eps}\Bigr), \omega\in \Omega$ have disjoint supports (since the function $\varphi\Bigl(\frac{\theta-\theta_\omega}{\eps}\Bigr)$ is supported 
in a ball of radius $\eps$ around $\theta_\omega$ and points $\theta_{\omega}, \omega \in \Omega$
are $2\eps$-separated). It follows that $f_l(\theta_\omega)= \omega_l \varphi(0)\eps^s, \omega\in \Omega, l=1,\dots, d$ and also that $\|f_l\|_{C^s}\leq 1$ (recall that $\|\tilde \varphi\|_{C^s}\leq 1$ and $\eps\leq 1/8$).

Define 
$$
\tau(\theta, \theta') := \biggl(\frac{1}{d}\sum_{l=1}^d (f_l(\theta)-f_l(\theta'))^2\biggr)^{1/2}, \theta, \theta'\in \Theta.
$$
We will need the following simple lemma:

\begin{lemma}
\begin{align}
\label{tau_norm}
\tau(\theta, \theta')= \frac{\varphi(0) \eps^{s-1}}{8} \|\theta-\theta'\|, \theta, \theta'\in \Theta.
\end{align}
\end{lemma}

\begin{proof}
Indeed, for all $\omega, \omega'\in \Omega,$ we have by a straightforward computation 
that 
\begin{align*}
\tau(\theta_{\omega}, \theta_{\omega'})= 2\varphi(0) \eps^s \sqrt{\frac{h(\omega, \omega')}{d}}
\end{align*}
(this is based on the fact that $f_l(\theta_{\omega})=\varphi(0) \eps^s \omega_l$). 
Combining this with \eqref{dist_theta_omega} yields 
\begin{align*}
\tau(\theta_{\omega}, \theta_{\omega'})= \frac{\varphi(0) \eps^{s-1}}{8} \|\theta_{\omega}-\theta_{\omega'}\|, \omega, \omega'\in \Omega,
\end{align*}
which implies the claim.

\qed
\end{proof}

In addition, we will use the following well known fact:

\begin{lemma}
If $\eps^2 \leq c' \sigma^2 d$ for a small enough numerical constant $c'>0,$
then 
\begin{equation}
\label{min_max_AAA}
\inf_{\hat \theta}\max_{\theta \in \Theta}{\mathbb E}_{\theta}\|\hat \theta(X)-\theta\|^2
\geq c'' \sigma^2 d,
\end{equation}
where the infimum is taken over all estimators $\hat \theta$ and 
$c''$ is a numerical constant.
\end{lemma}

The proof of this fact is quite standard 
(it could be based, for instance, on Theorem 2.5 in \cite{Tsybakov}).
Note that the lower bound could be also written as $c'' \eps^2$ for some 
numerical constant $c''.$

Suppose now that, for some $\delta>0,$ 
\begin{align}
\label{maxminmax}
\sup_{\|f\|_{C^s}\leq 1} \inf_{T} \sup_{\|\theta\|\leq 1} {\mathbb E}_{\theta}(T(X)-f(\theta))^2 < \delta^2. 
\end{align}
This implies that 
$$
\max_{1\leq l\leq d}\inf_{T} \max_{\theta \in \Theta} {\mathbb E}_{\theta}(T(X)-f_l(\theta))^2 < \delta^2
$$
and, moreover, for all $l=1,\dots, d$ there exist estimators $T_l(X)$ such that 
$$
\max_{\theta \in \Theta} {\mathbb E}_{\theta}(T_l(X)-f_l(\theta))^2 < \delta^2.
$$
It will be convenient to replace estimators $T_l(X)$ by estimators $\tilde T_l(X)$
defined as follows:
$\tilde T_l(X):= \eps^s \varphi(0)$ if $T_l(X)\geq 0$ and $\tilde T_l(X):= -\eps^s \varphi(0)$
otherwise. For these modified estimators, it is easy to check that  
\begin{align}
\label{bd_T_l}
\max_{\theta \in \Theta} {\mathbb E}_{\theta}(\tilde T_l(X)-f_l(\theta))^2 < 4\delta^2.
\end{align}
Define finally $\tilde \omega:= (\tilde \omega_1, \dots, \tilde \omega_d),$
where $\tilde \omega_l = \tilde \omega_l(X):={\rm sign}(\tilde T_l(X))$ and set 
$\tilde \theta=\tilde \theta (X):=\theta_{\tilde \omega}.$
The following identity immediately follows from the definitions and from \eqref{tau_norm}:
\begin{align*}
\|\tilde \theta-\theta\|= \frac{8}{\varphi(0)\eps^{s-1}}\tau(\tilde \theta, \theta)=
\frac{8}{\varphi(0)\eps^{s-1}}\biggl(\frac{1}{d}\sum_{l=1}^d (\tilde T_l(X)-f_l(\theta))^2\biggr)^{1/2}, 
\theta \in \Theta.
\end{align*}
Therefore, we can deduce from \eqref{bd_T_l}
\begin{align}
\label{buy_buy}
{\mathbb E}_{\theta}\|\tilde \theta-\theta\|^2  = \frac{8^2}{\varphi^2(0)\eps^{2(s-1)}}
\frac{1}{d} \sum_{l=1}^d {\mathbb E}_{\theta}(\tilde T_l(X)-f_l(\theta))^2
\leq \frac{4\ 8^2 \delta^2}{\varphi^2(0)\eps^{2(s-1)}}, \theta \in \Theta.
\end{align}
It remains to set $\eps^2 := c' (\sigma^2 d \wedge 1)$ and to use minimax lower bound \eqref{min_max_AAA}
to get 
\begin{align*}
\max_{\theta \in \Theta}{\mathbb E}_{\theta}\|\tilde \theta -\theta\|^2
\geq c'' \eps^2.
\end{align*}
Combining this with bound \eqref{buy_buy}, we get 
\begin{align*}
\frac{4\ 8^2 \delta^2}{\varphi^2(0)\eps^{2(s-1)}}
\geq c'' \eps^2,
\end{align*}
which implies that $\delta^2 \gtrsim \eps^{2s}.$ Therefore, 
\begin{align}
\label{maxminmax_BBB}
\sup_{\|f\|_{C^s}\leq 1} \inf_{T} \sup_{\|\theta\|\leq 1} {\mathbb E}_{\theta}(T(X)-f(\theta))^2 
\gtrsim (\sigma^2 d)^s \wedge 1. 
\end{align}

To complete the proof, it remains to show that, for some $c_2>0,$ the following bound holds:
\begin{align}
\label{bd_very_simple}
&
\sup_{\|f\|_{C^s}\leq 1} \inf_{T} \sup_{\|\theta\|\leq 1} {\mathbb E}_{\theta}(T(X)-f(\theta))^2 
\geq c_2 (\sigma^2 \wedge 1).
\end{align}
This easily follows from the bound of Theorem \ref{min_lower_bd}. To this end, take 
$f(\theta):= \langle \theta, u\rangle \varphi (\|\theta\|^2), \theta \in {\mathbb R}^d,$
where $\|u\|=\kappa$ for a small enough constant $\kappa>0$ and $\varphi : {\mathbb R}\mapsto [0,1]$ is a $C^{\infty}$
function with $\varphi(t)= 1, t\in [0,1]$ and $\varphi(t)= 0, |t|>2.$ It is easy to see that $u$ and $\varphi$
could be chosen in such a way that $\|f\|_{C^s}\leq 1.$ For such a function $f$ and for $\|\theta\|\leq 1,$
$\sigma_{f,\xi}^2(\theta)=\kappa^2\sigma^2$ and $K(f;\Sigma;\theta)\leq \frac{1}{\kappa}.$ 
Take also $\theta_0=0.$
The bound of Theorem \ref{min_lower_bd} now easily implies that, for small enough constants $c_3, c_4>0$
and for all $\sigma \leq c_3,$ 
\begin{align}
\label{bd_very_simple_A}
&
\inf_{T} \sup_{\|\theta\|\leq 1} {\mathbb E}_{\theta}(T(X)-f(\theta))^2 
\geq c_4 \sigma^2.
\end{align}
If $\sigma > c_3,$ then $\sigma^2 d\gtrsim c_3^2,$ and bound \eqref{maxminmax_BBB}
implies that, for some $c_4'>0,$ 
\begin{align*}
&
\inf_{T} \sup_{\|\theta\|\leq 1} {\mathbb E}_{\theta}(T(X)-f(\theta))^2 
\geq c_4'.
\end{align*}
Together with \eqref{bd_very_simple_A}, this implies \eqref{bd_very_simple}.

\qed
\end{proof}

{\bf Acknowledgement.} The authors are very thankful to Martin Wahl for careful reading of the paper and pointing out a number of typos and to anonymous referees for a number of useful suggestions.

\end{document}